\documentclass[review]{elsarticle}

\usepackage[english]{babel}
\usepackage[T1]{fontenc}
\usepackage{babelbib}
\usepackage{cancel}
\usepackage{amsmath}
\usepackage{amssymb}
\usepackage{indentfirst}
\usepackage[latin1]{inputenc}
\usepackage{geometry}
\usepackage{amsfonts}
\usepackage{graphicx}
\usepackage{float}
\usepackage{color}
\usepackage{verbatim}
\usepackage{enumerate}
 \usepackage[normalem]{ulem}
 \usepackage{cancel}
\usepackage{hyperref}
\usepackage{color}
\usepackage{tikz}
\usepackage{amsthm}

\journal{Journal of \LaTeX\ Templates}

\newcommand{\R}{{\mathbb R}}

\newcommand{\h}{{\Gamma}}   

\newcommand{\bc}{\begin{center}}
\newcommand{\ec}{\end{center}}

\newcommand{\GEN}[1]{\langle #1 \rangle}

\newcommand{\HQ}{{\mathbb H}}

\newcommand{\inv}{^{-1}}

\newcommand{\matriz}[1]{\begin{array} #1 \end{array}}

\newcommand{\BQ}{{\mathbb B}}
\newcommand{\KQ}{{\mathbb K}}
\newcommand{\Isom}{{\rm Isom}}
\newcommand{\XQ}{{\mathbb X}}
\newcommand{\SQ}{{\mathbb S}}
\newcommand{\te}{{\mathcal{T}}}

\newtheorem{theorem}{Theorem}[section]
\newtheorem{definition}[theorem]{Definition}
\newtheorem{lemma}[theorem]{Lemma}

\newtheorem{proposition}[theorem]{Proposition}
\newtheorem{remark}[theorem]{Remark}
\newtheorem{example}[theorem]{Example}

\begin{document}

\begin{frontmatter}

\title{Revisiting Poincar\'{e}'s Theorem on
presentations of discontinuous groups via fundamental polyhedra\tnoteref{mytitlenote}}
\tnotetext[mytitlenote]{The first author is supported in part by Onderzoeksraad of
Vrije Universiteit Brussel and Fonds voor
Wetenschappelijk Onderzoek (Flanders). The second author is supported by Fonds voor
Wetenschappelijk Onderzoek (Flanders)-Belgium.  The third author is partially supported by the Spanish Government under Grant MTM2012-35240 with "Fondos FEDER" and Fundaci\'{o}n S\'{e}neca of Murcia under Grant 04555/GERM/06 }

\author[mymainaddress]{E. Jespers}
\author[mymainaddress]{A. Kiefer}
\author[mysecondaryaddress]{\'A. del R\'io}

\address[mymainaddress]{Department of Mathematics, Vrije
Universiteit Brussel, Pleinlaan 2, 1050
Brussel, Belgium \\
emails: efjesper@vub.ac.be and akiefer@vub.ac.be}
\address[mysecondaryaddress]{Departamento de Matem\'{a}ticas,
 Universidad de Murcia,  30100 Murcia, Spain\newline
email: adelrio@um.es}

\begin{abstract}
We give a new self-contained  proof of Poincar\'e's Polyhedron Theorem on presentations of discontinuous groups of isometries  of a Riemann manifold of  constant curvature.
The proof is not based on the theory of covering spaces, but only makes use of basic geometric concepts.
In a sense one hence obtains a proof that is of a more constructive nature than most known proofs.
\end{abstract}


\end{frontmatter}


\section{Introduction}

Poincar\'e's Polyhedron Theorem is a widely known and often used result in mathematics. In summary, it states that if $P$ is a polyhedron with side pairing transformations satisfying several conditions, then the group $G$ generated by those side pairing transformations is discontinuous, $P$ is a fundamental polyhedron for $G$ and the reflection relations and cycle relations form a complete set of relations for $G$. Poincar\'e first published the theorem for dimension two in \cite{Poincare1882}, then one year later also for dimension three in \cite{Poincare1883}. A lot on this theorem may also be found in the literature, see for instance the books \cite{beardon, bridson, ElsGrunMen, Maskit, ratcliffe}. There are also various articles on this theorem, such as \cite{EpsteinPetronio, Maskit2, dRham, AnaninGrossi}.

In this paper we are  only interested in the presentation part of the theorem. It gives a method to obtain a presentation of a discontinuous subgroup of the group of isometries of a Riemann manifold of constant curvature from a fundamental polyhedron of the group (Theorem~\ref{PoincareT}).

Coming from the field of group rings, we are mainly interested in describing structures of unit groups of integral group rings of finite groups.
This unit group is a fundamental object in the study of the isomorphism problem of integral group rings; a standard reference on this topic is \cite{Sehgal}.
In some important cases, these unit groups may be described using some discontinuous groups of isometries of hyperbolic spaces \cite{PitadelRioRuiz2005,PitadelRio2006,JesPitRioRui2007}.
Hence our interest in the  presentation part of Poincar\'e's Theorem.

As already stated in \cite[Section 9]{EpsteinPetronio}, most of the published proofs on Poincar\'e's theorem are rather unsatisfactory (we refer the reader to \cite{EpsteinPetronio} for a long list of references on this topic). The two original versions {written by  Poincar\'e (\cite{Poincare1882} and \cite{Poincare1883})} are very complicated to read. Moreover, for the proof of the three-dimensional case one simply refers for a large part to the two-dimensional case, and this without really proving that the results used are still valid. More modern versions often only deal with dimension two, which simplifies the proof a lot and cannot be directly generalized to dimension three. Other modern versions are either very complicated too read or stay  hazy for some aspects.
The proofs in \cite{Maskit} and \cite{EpsteinPetronio} focus mostly on that part of  Poincar\'e's Theorem which states that if a polyhedron satisfies some conditions then it is a fundamental domain of a discontinuous group. As a consequence the presentation part of the theorem is obtained somehow indirectly and the intuition on the presentation part is hidden in these proofs. In \cite{macbeath}, an analogous, but different, presentation result is proven for topological transformation groups. The generators obtained by Poincar\'e's Polyhedron Theorem can be derived from this result. It should also be possible to derive the relations given by Poincar\'e, but here a lot more work has to be done. Moreover, geometrically seen, Poincar\'e's Theorem is more intuitive.
That is why, in this paper, we give a  new and  self-contained proof of the  presentation part of Poincar\'e's Theorem, which we hope   to contribute in  a deeper understanding of this important result.

In Sections 2 and 3 we give the necessary definitions and background on  well known facts
on geometry,  hyperplanes and polyhedra. Section 4 discusses tessellations as a preparation to Section 5 where the main theorem  on presentations (Theorem~\ref{PoincareT}) is proven. In the first part of this section we prove that the  side paring transformations are generators and in the second part we prove that the pairing and cycle relations form a complete set of relations.  To guide the reader through this long section, we will include at the beginning of these two parts a brief outline of the main idea of the proofs.

\section{Background} \label{SectionHyperbolicSpaces}

We use the standard topological notation for  a subset $X$ of a topological space:
    \begin{eqnarray*}
    \overline{X}&=&\text{Closure of } X; \\
    X^0&=&\text{Interior of } X; \\
    \partial{X}&=&\text{Boundary of } X.
    \end{eqnarray*}

In the remainder of the paper $\XQ$ denotes an $n$-dimensional Riemann manifold of constant curvature.
In other words, $\XQ$ is either the $n$-dimensional Euclidean space $\R^n$, the $n$-dimensional unit sphere $\SQ^n$ of $\R^{n+1}$ or one of the models of $n$-dimensional hyperbolic space.
The curvature of $\R^n$  is 0,  that of $\SQ^n$ is positive and  the hyperbolic space  has negative  curvature. We assume that the reader is familiar with these three spaces.
Standard references on this topic are  \cite{beardon,BP91,bridson,ElsGrunMen,gromov,mac,Maskit,ratcliffe}.
In this section, we recall some basic facts on the geometry of $\XQ$.
We restrict to the background required for our purposes.

We will use three models for the hyperbolic $n$-space. The first one is the Poincar\'{e} upper half-space
	$$\HQ^n = \{(x_1,\dots,x_n) \in \R^n : x_n>0\},$$
with the metric $d$ given by
	$$\cosh d(a,b) = 1+\frac{\|a-b\|^2}{2a_nb_n},$$
where $\|\;\|$ denotes the Euclidean norm and $a=(a_1,\dots,a_n),\; b=(b_1,\dots,b_n)\in \HQ^n$.
In particular, the hyperbolic ball
$$B_{\HQ^n}(a,r)=\{ x\in\HQ^{n} \; : \; d (x,a)\le r\}$$
with hyperbolic center $a=(a_{1}, \ldots , a_{n})$ and hyperbolic radius $r$ is the Euclidean ball  given by
\begin{eqnarray} \label{HypSpereEucSpereEq}
(x_{1}-a_{1})^{2}+ \cdots + (x_{n-1}-a_{n-1})^{2} + (x_{n} -a_{n}\cosh (r) )^{2} &\leq & (a_{n}\sinh (r))^{2}.
\end{eqnarray}
In other words
	$$B_{\HQ^n}(a,r)=B_{\R^n}((a_1,\dots,a_{n-1},a_n\cosh(r)),a_n\sinh(r)).$$
Hence, the topology of $\HQ^n$ is that induced by the Euclidean topology of $\R^n$ and a subset of $\HQ^{n}$ is compact if and only if it is closed and bounded.

The second model for the hyperbolic $n$-space is the open \emph{unit ball}
	$$\BQ^n=\{a\in \R^n  :  \| a \|<1\}$$
with the metric $d$ given by
	$$\cosh d(a,b)=1 + 2\frac{\| a-b\|^2}{(1-\| a \|^2)(1-\| b \|^2)}.$$

The third model is the Klein model $\KQ^n$ whose underlying set also  is  the open unit ball but the metric $d$ is given by
$$d(a,b)=\frac{1}{2}\ln \frac{\|a-b'\|\|b-a'\|}{\|a-a'\|\|b-b'\|},$$
where the definition of the points $a'$ and $b'$ is depicted in Figure~\ref{Klein model}.
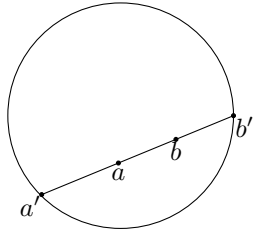
\begin{figure}[h!]
\begin{center}
\begin{tikzpicture}[scale=1.5]
\draw (0,0) circle (1cm);
\draw[-] (1,0) -- (-0.7,-0.7);
\foreach \y in {0,0.3,0.6,1}
{
\filldraw (1-1.7*\y,-0.7*\y) circle (0.5pt);
}
\draw (1.1,-0.1) node {$b'$};
\draw (1-1.7*0.3,-0.1-0.7*0.3) node {$b$};
\draw (1-1.7*0.6,-0.1-0.7*0.6) node {$a$};
\draw (1-1.7-0.1,-0.1-0.7) node {$a'$};
\end{tikzpicture}
\caption{\label{Klein model}
Distance in Klein model. 
}
\end{center}
\end{figure}

The boundary in $\R^n$ of $\HQ^n$ is $\partial \HQ^n=\{(a_1,\dots,a_{n-1},0):a_1,\dots,a_{n-1}\in \R\}$ and the boundary of $\BQ^n$ and $\KQ^n$ is $\partial \BQ^n=\{a\in \R^n : \|a\|=1\}$.
The geodesics in $\HQ^n$ and $\BQ^n$ are the intersection with Euclidean lines and circles orthogonal to the boundary and the geodesics in $\KQ^n$ are the intersection of Euclidean lines with the unit ball (Figure~\ref{Geodesics}).
\begin{figure}[h!]
\begin{center}
\begin{tikzpicture}[scale=0.7]
\clip (-5,0) rectangle (5,4);
\draw[-] (1,0) -- (1,4);
\draw[-] (-5,0) -- (5,0);
\draw (0,0) circle (2);
\draw (2,0) circle (2);
\draw (-1.5,0) circle (1.5);
\draw (0.5,0) circle (0.5cm);
\end{tikzpicture}
\hspace{1cm}
\begin{tikzpicture}[scale=1.2]
\clip[draw] (0,0) circle (1);
\draw[-] (2,-1) -- (-2,1);
\draw[-] (0,-1) -- (0,1);
\draw (0,2.23) circle (2);
\draw (2.23,0) circle (2);
\draw (-1,-1) circle (1);
\end{tikzpicture}
\hspace{1cm}
\begin{tikzpicture}[scale=1.2]
\clip[draw] (0,0) circle (1);
\draw[-] (2,-1) -- (-2,1);
\draw[-] (0,-1) -- (0,1);
\draw[-] (3,-1) -- (-2,1);
\draw[-] (1,0) -- (0,-1);
\end{tikzpicture}
\caption{\label{Geodesics}
Geodesics in $\HQ^2$, $\BQ^2$ and $\KQ^2$.}
\end{center}
\end{figure}

If $x$ is a point of $\SQ^n$ then by $x'$ we denote its  \emph{antipode}.
If $x$ and $y$ are two different points of $\XQ$ then there is a unique geodesic containing $x$ and $y$ unless $\XQ$ is spherical and $x$ and $y$ are antipodal, in which case all the geodesics containing $x$ contain $y$ too.

In the language of Riemannian geometry, a \emph{subspace} of $\XQ$ is a complete totally geodesic submanifold.
If $\XQ=\R^n$ then the subspaces are the affine varieties.
If $\XQ=\SQ^n$ then the subspaces are the intersection with $\XQ$ of the affine subspaces of $\R^{n+1}$ containing $0$.
In $\HQ^n$ and $\BQ^n$, the subspaces are intersections with $\XQ$ of affine subspaces and half-spheres orthogonal to the boundary of $\XQ$.
Finally, the subspaces of $\KQ^n$ are the intersections of $\KQ^n$ with subspaces of $\R^n$.
Every subspace $S$ of $\XQ$ is a Riemann manifold with the same constant curvature than $\XQ$.
The codimension of $S$ in $\XQ$ is the difference between the dimensions of $\XQ$ and $S$.
For every non-empty subset $U$ of $\XQ$ there is a unique minimal subspace of $U$ containing $U$, it is called the \emph{subspace generated} by $U$.
If $S$ is a subspace of $\XQ$ of dimension $k$ and $x\in \XQ\setminus S$ then the subspace generated by $S\cup \{x\}$ is of dimension $k+1$.

Let $x$ and $y$ be two different points of $\XQ$ and assume that $x$ and $y$ are not antipodal in case $\XQ$ is spherical.
Then, $[x,y]$ denotes the geodesic of $\XQ$ joining $x$ and $y$ and $(x,y)=[x,y]\setminus \{x,y\}$.
A subset $S$ of $\XQ$ is said to be convex if $[x,y]\subseteq S$ for all distinct non-antipodal $x,y\in S$.
Let $x\in \XQ$ and let $U$ be a subset of $\XQ$, such that if $\XQ$ is spherical then $U$ does not contain the antipode of $x$.
Then the \emph{cone} based on $U$ with vertex $x$ is $\bigcup_{u\in U} [x,u]$.

\section{Hyperplanes, half-spaces and polyhedra}

In this section there is quite a bit of overlap with the beginning of \cite{EpsteinPetronio}.
For the sake of clearness and completeness, we reprove all the lemmas  needed for this paper.
As our definition of cell is  different from the one given in \cite[Definition 2.8]{EpsteinPetronio}, from Section 4 onwards both papers are independent.
Moreover, the proof of lemma~\ref{IntersectionTwoL} is more complete than the one given in \cite{EpsteinPetronio}.

A hyperplane is a codimension $1$ subspace.
If $H$ is a hyperplane of $\XQ$ then $\XQ\setminus H$ has two connected components called the \emph{open half-spaces} defined by $H$.
If $U$ is one of these then the other open half-space defined by $H$ is denoted $U'$ and
we have
\begin{eqnarray*}
H&=&\partial U = \partial \overline{U} = \partial U' = \partial \overline{U'}, \\
\overline{U}&=&U\cup H, \\
\overline{U}^\circ&=&U, \\
U'&=&\XQ \setminus \overline{U} \text{ and} \\
\overline{U}&=&\XQ \setminus U'.
\end{eqnarray*}
The sets $\overline{U}$ and $\overline{U'}$ are called the closed half-spaces defined by $H$.
Moreover, if $Z$ is one of the two closed half-spaces defined by $H$ then the other is denoted $Z'$.

It is easy to see that the set formed by the non-empty intersections of finitely many open half-spaces is a basis for the topology of $\XQ$.

\begin{lemma}\label{ConeOpenL}
Let $x\in \XQ$ and $U\subseteq \XQ$ and assume that one of the following conditions holds.
\begin{enumerate}
\item $U$ is an open subset of a hyperplane $H$ of $\XQ$ and $x\not\in H$.
\item $U$ is an open subset of $\XQ$ and if $\XQ$ is spherical then $U$ does not contain the antipode of $x$.
\end{enumerate}
Then $\bigcup_{u\in U} (x,u)$ is an open subset  of  $\XQ$ and it is dense in the cone based on $U$ with vertex $x$.
\end{lemma}

\begin{proof}
(1) Assume that $U$ is an open subset of a hyperplane $H$ and $x\not\in H$.
First observe that, in the spherical case, $H$ is closed under taking antipodes and therefore it does not contain the antipode of $x$.
Let $C=\cup_{u\in U} (x,u)$. We have to prove that $C$ is open.
We first reduce the statement to the Euclidean case.
This is clear for the hyperbolic geometry by using the Klein model.
To reduce the spherical case to the Euclidean case, consider $\SQ^n$ as a subset of $\R^{n+1}$ and let $V$ be the half space of $\SQ^n$ with border $H$ and containing $x$ and let  $E$ be the hyperplane of $\R^{n+1}$ tangent to the sphere in the point of $V$ whose tangent in $\SQ^n$ is parallel to $H$.
The stereographic projection from the center of the sphere is a bijection $V\rightarrow E$ mapping the intersections of $V$ with the geodesics of the sphere to the Euclidean geodesics of $E$. Hence the statement for the Euclidean geometry implies the statement for the hyperbolic and spherical geometries.

So we only have to prove the statement for $\XQ=\R^n$.
Making use of some affine transformations we may, without loss of generality, assume that $x$ is the origin and $H$ is given by the equation $x_n=1$. As $U$ is a union of squares of the form $(a_1,b_1)\times \dots \times (a_{n-1},b_{n-1})\times \{1\}$ it is enough to prove the statement under the assumption that $U$ is one of these squares. Hence, again making use  of some linear transformations if needed, we may assume that $U=(-1,1)^{n-1}\times \{1\}$. Then $C=\{(x_1,\dots,x_n) : 0<x_n<1 \text{ and } |x_i|<x_n \text{ for each } 1<i<n\}$, which is clearly an open subset.

(2) Assume now that $U$ is open and, in case $\XQ$ is spherical it does not contain the antipode of $x$.
Every  $u\in \XQ\setminus \{x\}$ belongs to some hyperplane $H_u$ such that $x\not\in H_u$. Thus, by part (1), $\cup_{u\in U} (x,u) = \cup_{u\in U} \cup_{v\in U\cap H_u} (x,v)$ is open and its closure contains $U\cup \{x\}$.
So,  it is dense in the cone based on $U$ with vertex $x$.
\end{proof}

\begin{lemma}\cite[Lemma~2.3]{EpsteinPetronio}\label{InteriorLemmaL}
If $D$ is an intersection of closed half-spaces in $\XQ$ then either $D^\circ\ne\emptyset$  or $D$ is contained in a hyperplane of $\XQ$.
Moreover, if $D^\circ\ne\emptyset$ then $D^\circ$ is dense in $D$.
\end{lemma}
\begin{proof}
We may assume that $D$ is not empty.
Let $\mathcal{S}$ be the set whose elements are  the subspaces $S$ of $\XQ$ with the property that  $D\cap S$ has a non-empty interior, say  $V$, as a subset of $S$ and such that $V$ is dense in $D\cap S$.
Clearly, if $x\in D$ then $\{ x \}\in \mathcal{S}$.
So $\mathcal{S}\ne \emptyset$.
Let $S$ be a maximal element of $\mathcal{S}$.
It is enough to show that $D\subseteq S$.
Indeed,  if this holds then either $S$ is contained in a hyperplane, and hence so is $D$, or $\XQ=S\in \mathcal{S}$ and thus $D^\circ$ is dense in $D$.

Assume that $D\not\subseteq S$ and let $x\in D\setminus S$.
Let $V$ be the interior of $D\cap S$ considered as a subset of $S$.
By definition, $V$ is a non-empty open subset of $S$.
Let $T$ be the subspace generated by $S\cup \{x\}$. Then $S$ is a hyperplane of $T$.
By Lemma~\ref{ConeOpenL}, $C_x=\cup_{v\in V} (x,v)$ is an open subset of $T$ contained in $D$ and $x\in \partial C_x$.
As this property holds for every $x\in (D\cap T) \setminus S$, we obtain that
$\cup_{x\in (D\cap T) \setminus S} C_x$ is open in $T$ and dense in $D\cap T$.
Therefore $T\in \mathcal{S}$, contradicting the maximality of $S$.
Thus $D\subseteq S$, as desired.
\end{proof}

A set $\h$ of subsets of $\XQ$ is  said to be \emph{locally finite} if each point in $\XQ$ has a neighborhood that intersects only finitely many of the sets in $\Gamma$.

A  non-empty subset $P$ of $\XQ$ is said to be a \emph{polyhedron of $\XQ$} if $P=\cap_{Z\in \h} Z$ for a family $\h$ of closed half-spaces such that $\{\partial Z : Z\in \h\}$ is locally finite.
In this case, one says that $\h$ defines the polyhedron $P$.
For example, a subspace $S$ is a polyhedron because $S=\cap_{i=1}^k H_i$ for some hyperplanes $H_1,\dots,H_k$  and then $S=\cap_{i=1}^k Z_i\cap Z'_i$ where $Z_i$ and $Z'_i$ are the two closed subspaces containing $H_i$.
Let $P$ be a polyhedron and let $S$ be the subspace generated by $P$.
By Lemma~\ref{InteriorLemmaL}, $P$ contains a non-empty  open subset of $S$ (and it is dense in $P$).
We then say that $P$ is \emph{thick in $S$} (note that  $S$ is the unique subspace of $\XQ$ in  which $S$ is thick).
If $P$ is thick in $\XQ$ we simply say that $P$ is \emph{thick}.
The \emph{dimension} (respectively, \emph{codimension}) of $P$ is by definition the dimension  of $S$ (respectively, the codimension of $S$ in $\XQ$).
The relative interior of $P$, denoted $P^r$, is the interior of $P$ in the subspace generated by $P$.

\begin{lemma} \cite[Proposition~2.3]{EpsteinPetronio}\label{EssentialBoundaryL}
Let $P$ be a thick polyhedron and let $\h$ be a family of closed half-spaces defining $P$.
Then $\partial P = \bigcup_{Z\in \h} P\cap \partial Z$ and the following conditions are equivalent for a closed half-space $Z_0$ of $\XQ$:
\begin{enumerate}
 \item $P\ne \bigcap_{Z\in \h\setminus \{Z_0\}} Z$.
 \item $P\subseteq Z_0$ and $P\cap \partial Z_0$ is thick in $\partial Z_0$.
\end{enumerate}
\end{lemma}
\begin{proof}
The inclusion $\partial P \supseteq \bigcup_{Z\in \h} P\cap \partial Z$ is clear. For the converse inclusion assume that $x\in P\setminus \partial Z$ for every $Z\in \h$.
As $\{\partial Z : Z\in \h\}$ is locally finite, any ball of $\XQ$ with center $x$ intersects only finitely many $\partial Z$ with $Z\in\h$ and hence $x$ has a neighborhood not intersecting any $\partial Z$. This implies that $x\in P^\circ$ and thus $x\not\in \partial P$. Since $P$ is closed it follows that  $\partial P = \bigcup_{Z\in \h} P\cap \partial Z$.

Let $P_0=\cap_{Z\in \h\setminus \{Z_0\}} Z$.

(1) implies (2). Assume that $P\ne P_0$. Clearly $Z_0\in \h$ and therefore $P\subseteq Z_0$.
Let $x\in P_0\setminus P$. As $P$ is thick there is a non-empty open subset $U$ of $\XQ$ contained in $P$ such that if $\XQ$ is spherical then the antipode of $x$ is not in $U$.
Then $x\in \XQ\setminus Z_0=({Z_0'})^\circ$ and therefore every open segment $(x,u)$ with $u\in U$ intersects $\partial Z_0$.
By Lemma~\ref{ConeOpenL},  $C=\cup_{u\in U} (x,u)$ is an open subset of $\XQ$ contained in $P_0$ and hence $C\cap \partial Z_0$ is a non-empty open subset of $\partial Z_0$ contained in $P$.
Therefore $P\cap \partial Z_0$ is thick in $\partial Z_0$.

(2) implies (1). Assume that  $Z_0$   satisfies (2). Let $x$ be an element of the interior of $P\cap \partial Z_0$ in $\partial Z_0$.
It easily is verified that then $x\not\in \partial Z$ for every $Z\in \h\setminus \{Z_0\}$.
Hence, $x\in (P_0)^{\circ}$.
Clearly $x\not\in (Z_{0})^{\circ}$ as $x\in \partial Z_{0}$. Because $P^{0}\subseteq (Z_0)^{0}$  it follows that $x\not\in P^{\circ}$.
Thus $P\ne P_0$.
\end{proof}

Observe that condition (2) of Lemma~\ref{EssentialBoundaryL} does not depend on $\h$, but only depends on $P$.
A closed half-space  $Z_0$ of $\XQ$ satisfying the equivalent conditions of Lemma~\ref{EssentialFacesL}  is called an essential half-space
of $P$ and $\partial Z_0$ is called an essential hyperplane of $P$.

\begin{lemma}\cite[Lemma~2.4 and Proposition~2.5]{EpsteinPetronio}  \label{EssentialFacesL}
Every thick polyhedron  of $\XQ$  is the intersection of its essential half-spaces  and, in particular,
$\partial P$ is the union of the intersection of $P$
with the essential hyperplanes of $P$.
\end{lemma}
\begin{proof}
Let $P$ be a thick polyhedron of $\XQ$ and let $\h$ be a set of closed half-spaces defining $P$.
Let $\h_1$ be the set of essential closed half-spaces of $P$ and let $P_1=\cap_{Z\in \h_1} Z$.
As $\h_1\subseteq \h$ we have that  $P\subseteq P_1$.
Assume that this inclusion is strict and take $x\in P_1\setminus P$.
Let $Z_1,\dots,Z_k$ be the elements of $\h$ whose boundaries  contain $x$.
Then, there is a ball $U$ of $\XQ$ centered in $x$ such that $U\cap \partial Z=\emptyset$ for every $Z\in \h\setminus \{Z_1,\dots,Z_k\}$.
Let $l$ be a non-negative integer such that $l\leq k$ and
$Z_i$ is an essential half-space of $P$ if and only if $i\le l$.
Then $P=\cap_{Z\in \Gamma\setminus \{Z_{l+1},\dots,Z_k\}} Z$ and hence $x\in  \cap_{i=1}^l Z_i\cap U=P_1\cap U =\cap_{Z\in \Gamma\setminus \{Z_{l+1},\dots,Z_k\}} Z \cap U = P\cap U$, a contradiction.

The last part of the statement of the lemma follows from Lemma~\ref{EssentialBoundaryL}.
\end{proof}

\begin{lemma}\cite[Lemma~2.7]{EpsteinPetronio} \label{IntersectionTwoL}
Let $Z_1, \, Z_{2}$ and $Z_3$ be closed half-spaces of $\XQ$ such that $\partial Z_1 \cap \partial Z_2 \cap \partial Z_3$ has  codimension $2$ and
$Z_1 \cap Z_2 \cap Z_3$ is thick. Then $Z_1 \cap Z_2 \cap Z_3 = Z_i \cap Z_j$ for some $i,j\in \{ 1,2, 3\}$.
\end{lemma}

\begin{proof}
We may assume that $Z_1$, $Z_2$ and $Z_3$ are pairwise different.
Then $\partial Z_1$, $\partial Z_2$ and $\partial Z_3$ are pairwise different for otherwise $Z_1\cap Z_2\cap Z_3$ is not thick.

We first prove the result for $\XQ=\R^n$, the Euclidean space.
Then each $\partial Z_i$ is an Euclidean hyperplane in $\R^n$
and   $\partial Z_1 \cap \partial Z_2 \cap \partial Z_3$ is a codimension $2$ affine subspace of $\R^n$.
Applying some  Euclidean transformation if needed,  we may assume that
   $Z_1 = \{ (x_1, x_2, \ldots , x_n) \in \R^n \; : \; x_1\geq 0\}$ and
   $Z_2 = \{ (x_1, x_2, \ldots , x_n) \in \R^n \; : \; x_2\geq 0\}$.
Then,
  $\partial Z_1 \cap \partial Z_2 \cap \partial Z_3 = \{ (x_1, x_2, \ldots , x_n) \in \R^n \; : \; x_1=x_2=0\}$ and
  $\partial Z_3 = \{ (x_1, x_2, \ldots , x_n) \in \R^n \; : \;  a_1 x_1 + a_2 x_2 =0\}$
with
  $a_1 a_2 \neq 0$ and $a_1 >0$.

Assume  $a_2 >0$. If
  $Z_3  = \{ (x_1, x_2, \ldots , x_n) \in \R^n \; : \; a_1 x_1 + a_2 x_2 \leq 0\} $ then $Z_1 \cap Z_2 \cap Z_3 \subseteq
    \partial Z_1 \cap \partial Z_2$
contradicting the thickness of $Z_1 \cap Z_2 \cap Z_3$. So,
   $Z_3 = \{ (x_1, x_2, \ldots , x_n) \in \R^n \; : \; a_1 x_1 + a_2 x_2 \geq  0\} $
and $Z_1 \cap Z_2 = Z_1 \cap Z_2 \cap Z_3 $.

To finish the proof for $\XQ=\R^{n}$, it remains to deal with   $a_2 <0$.
If
   $Z_3 = \{ (x_1, x_2, \ldots , x_n) \in \R^n \; : \; a_1 x_1 + a_2 x_2 \leq 0\}$ then $Z_1 \cap Z_3 = Z_1 \cap Z_2 \cap Z_3$.
Otherwise,  $Z_3 = \{ (x_1, x_2, \ldots , x_n) \in \R^n \; : \; a_1 x_1 + a_2 x_2 \geq 0\}$ and thus
   $Z_2 \cap Z_3 = Z_1 \cap Z_2 \cap Z_3$.
This finishes the proof in the Euclidean case.

In case $\XQ=\SQ^n\subseteq \R^{n+1}$ each $Z_i=\SQ^n\cap Y_i$ with $Y_i$ a closed half-space of $\R^{n+1}$ such that $0\in \partial Y_i$.
As $Z_1\cap Z_2 \cap Z_3$ is thick in $\SQ^n$ and $Y_1\cap Y_2 \cap Y_3$ contains the Euclidean cone with center $0$ and base $Z_1\cap Z_2\cap Z_3$, we deduce that $Y_1\cap Y_2 \cap Y_3$ is thick in $\R^{n+1}$.
Then, from the Euclidean case we deduce that $Y_1\cap Y_2 \cap Y_3=Y_i\cap Y_j$ for some $i,j\in \{1,2,3\}$ and hence $Z_1\cap Z_2 \cap Z_3=Z_i\cap Z_j$.

To prove the result in the hyperbolic case we use the Klein model $\KQ^n$ seen as subset of $\R^n$.
Then the hyperplanes are the intersection of Euclidean hyperplanes with $\KQ^n$ and the result follows again from the Euclidean case.
\end{proof}

\begin{lemma}\label{ConnectedCod2L}
 Let $\h$ be a countable set of proper subspaces of $\XQ$. Then
\begin{enumerate}
 \item $\XQ \ne \bigcup_{S\in \h} S$.
 \item If each  $S\in \h$ has   codimension  at least $2$  then for   $x,y\in \XQ \setminus \cup_{S\in \Gamma} S$ there is $z\in \XQ\setminus \{ x,y,x',y'\}$ such that $([x,z]\cup [z,y])\cap \cup_{S\in \Gamma} S=\emptyset$. In particular,  $\XQ \setminus \cup_{S\in \h} S$ is path  connected.
\end{enumerate}
\end{lemma}

\begin{proof}
(1) If $S \in \Gamma$, the complement of $S$ is a dense open set. Hence the complement of $\bigcup_{S\in \h} S$ is a countable intersection of dense open subsets of $\XQ$. By Baire's category theorem, this intersection is dense and hence non-empty. Thus the result follows. 

(2) For  $S\in \h$ and $x\in \XQ$ let $S_x$ denote the subspace of $\XQ$ generated by $S\cup \{x\}$.
Because of the assumption, each $S_x$ is a proper subspace of $\XQ$.
Assume  $x,y\in \XQ \setminus \bigcup_{S\in \h} S$.
By (1), there exists  $z\in \XQ \setminus \bigcup_{S\in \h} (S_x \cup S_y)$.
If $u\in (x,z)\cap S$ for some $S\in \h$ then $x$ and $u$ are different and non-antipodal points in $S_x$ and hence the geodesic containing both $x$ and $u$ is contained in $S_x$, contradicting the fact that $z\not\in S_x$. Therefore, the concatenation of the segments $[x,z]$ and $[z,y]$ is a path joining $x$ and $y$ contained in $\XQ\setminus \bigcup_{S\in \h} S$.
Hence, (2) follows.
\end{proof}

\section{Tessellations}

A \emph{tessellation} of $\XQ$ is a set $\te$ consisting of thick polyhedra of $\XQ$ such that the following properties are satisfied:
 \begin{enumerate}
 \item $\XQ=\cup_{P\in \te} P$, and
 \item $P^0\cap Q^0=\emptyset$ for every two different members $P$ and $Q$ of $\te$.
 \end{enumerate}
If only the second condition is satisfied then we call $\te$ a partial tessellation of $\XQ$.
 The members of a partial tessellation are called \emph{tiles}.
It is easy to see that $P\cap Q^0=\emptyset$  for any two distinct tiles $P$ and $Q$. In particular,
   $P\cap Q = \partial P  \cap \partial Q$.
All the tessellations and partial tessellations that will show up will be locally finite, that is every compact subset intersects only finitely many tiles.
One readily verifies that  every locally finite partial tessellation has to be countable.
If $T$ is a tile of a tessellation $\te$ then
     $\partial T = \cup_{R\in (\te \setminus \{ T\} )} T\cap R$.
This is not necessarily true if $\te$ is a partial tessellation.

\begin{definition}\label{CellD}
A \emph{cell} $C$  of a  partial  tessellation  $\te$ of $\XQ$ is a non-empty  intersection of tiles satisfying the following property:
if $T\in \te$ then either $C\subseteq T$ or $C^r\cap T=\emptyset$.
\end{definition}

Clearly every cell of a locally finite partial tessellation $\te$ is a polyhedron and it is contained in only finitely many tiles of $\te$.
As the intersection of two different tiles is contained in the boundary of both, the cells of codimension 0 are precisely the tiles and hence the codimension of the intersection of two different tiles is at least 1.
By definition, a  \emph{side} of $\te$ is a cell of codimension 1 and an  \emph{edge} of $\te$  is a cell of codimension 2.
If $T$ is a tile of $\te$ and $C$ is a cell (respectively, side, edge) of $\te$ contained in $T$ then we say that $C$ is a cell (respectively, side, edge) of $T$ in $\te$. In case the tessellation is clear from the context we simply say that $C$ is a cell, side or edge of $T$.

\begin{lemma}\label{CellGeneratedL}
Let $\te$ be a locally finite tessellation of $\XQ$. If $x\in \XQ$ and $C=\bigcap_{T\in \te,x\in T} T$
then $C$ is a cell of $\te$ and $x\in C^r$.
\end{lemma}

\begin{proof}
Let $x\in \XQ$ and  let $T_1,\dots,T_k$ be the tiles of $\te$ containing $x$.
Hence,
    $C=T_1\cap \dots \cap T_k$ and $T_1,\dots,T_k$
are the only tiles containing $C$.

To  prove that  $C$ is a cell we need to show that  if $Q$ is a tile different  from any $T_i$ then $Q\cap C^r=\emptyset$.
We first consider the case where $\XQ$ is spherical and $x'\in C$. In this case we prove that $\XQ=T_1\cup \dots \cup T_k$, which of course implies the desired statement.
Let $U$ be an open convex neighborhood of $x$ such that $U\cap T=\emptyset$ for every $T\in \te\setminus \{T_1,\dots,T_k\}$.
Let  $z\in \XQ$.
If $z=x$ or $x'$ then, by assumption, $z \in T_i$ for some (all) $i$.
Otherwise the geodesic containing $x'$ and $z$ also contains  $x$ and therefore it intersects $U\setminus \{x\}$.
In fact there exist non-antipodal elements $x_1$ and $x_2$ in $U$
that both  belong to the complete geodesic containing $x$ and $z$ and are such that
$x\in (x_1,x_2)$ and $z\not\in (x_1,x_2)$.
Then $z$ belongs to either $[x',x_1]$ or $[x',x_2]$.
By symmetry we may assume that $z\in [x',x_1]$. Moreover, as $x_1\in U \subseteq T_1\cup \dots \cup T_k$, there is $i=1,\dots,k$ with $x_1\in T_i$.
Then $z\in T_i$, as desired.

So we may assume that either $\XQ$ is not spherical or $x'\not\in C$ and we argue by contradiction.
Thus,  suppose that there exist $Q\in \te\setminus \{T_1,\dots,T_k\}$ and $y\in Q\cap C^r$.
In particular, $y\ne x$ and if $\XQ$ is spherical then $y\ne x'$.
Let $U$ be an open convex neighborhood of $x$ such that $U\cap T=\emptyset$ for every $T\in \te\setminus \{T_1,\dots,T_k\}$ and $y\not \in U$.
Hence, we can take the geodesic $g$ containing $x$ and $y$ and take a point
$y_1\in C^r$ such that $y\in (x,y_1)$.
This point exists because $g$ is contained in the subspace generated by  $C$ and hence $y$ is an interior point of $g\cap C$.
By Lemma~\ref{ConeOpenL}, $W=\cup_{u\in U} (y_1,u)$ is an open subset of $\XQ$.
Since $y\in W\cap Q$ and $Q$ is thick, we get that $W$ contains a point $z\in Q^\circ$.
Let $u\in U$ be such that  $z\in (y_1,u)$.
As $U\subseteq T_1\cup \dots \cup T_k$, $u\in T_i$ for some $i$, $(y_1,u)\subseteq T_i$ and therefore $z\in T_i\cap Q^\circ$.
However, this contradicts with the fact that $T_i$ and $Q$ are different tiles of the the tessellation $\te$.
So, in this case, $C$ indeed is a cell.

To prove the second part, assume that $x\not\in C^r$ and let $L$ be the subspace generated by $C$.
Clearly, the dimension of $L$ is positive and $k>1$.
Therefore $C\subseteq \partial T_i$ for every $i$.
Consider $C$ as a thick polyhedron of $L$.
As $x\not\in C^r$, it follows from Lemma~\ref{EssentialBoundaryL}, that
 $x$ belongs to one of the essential hyperplanes of $C$, as thick
polyhedra of $L$.
Fix $y\in C^r$ and an open interval $(y,z)$ containing $x$.
Then $(y,x)\subseteq C^r$ and $(x,z)\cap C=\emptyset$.
Therefore $z\not\in T_i$ for some $i$.
Renumbering the $T_i$'s and replacing $(x,z)$ by a smaller interval if necessary, one may assume that $(x,z)\cap T_1=\emptyset$.
We claim that $H\cap [y,z]=\{x\}$ (equivalently $y\not\in H$) for some essential hyperplane $H$ of $T_1$.
Otherwise $y$ belongs to all the essential hyperplanes of $T_1$ containing $x$.
Then $(x,z)$ is contained in all these essential hyperplanes.
If $V$ is an open neighborhood of $x$ only intersecting the essential hyperplanes of $T_1$ containing $x$ then $(x,z)\cap V$ is a non-empty subset contained in $T_1$, contradicting the construction.
This proves the claim.
Note that $x\in U\cap H\cap T_{1}$. Hence, $U\cap H\cap T_{1}$ is a non-empty open subset  of $H\cap T_{1}$.
So, by the second part of Lemma~\ref{InteriorLemmaL}, there exists $w\in (U\cap H \cap T_{1})\cap (H\cap T_{1})^{r}$.
Hence,  $H$ is the only essential hyperplane of $T_1$ containing $w$.

We claim that $(w,y)\subseteq T_{1}^{\circ}$.
Indeed, for suppose the contrary, then there exists $u\in (w,y) \cap \partial T_{1}$.
So $u\in H_{1}$ for some essential hyperplane $H_{1}$ of $T_{1}$.
If $H_{1}\neq H$ and $Z_1$ is the closed half-space of $\XQ$ with $\partial Z_1= H_{1}$ and $T_{1}\subseteq Z_1$ then $w,y\in Z_{1}^\circ$.
Then $u\in (w,y)\subseteq Z_1^\circ$, a contradiction.
So, $H_{1}=H$ and $u\in (w,y)\cap H$.
Now $y\not\in H$ and $w\in T_{1}$.
Hence a reasoning as above  (interchanging the role of $w$ and $y$ and replacing
$H_{1}$ by $H$) yields that $(w,y)\cap H =\emptyset$, a contradiction.

Because of the claim and since $w\in U$ and $U$ is open in $\XQ$, there exists $z_{1}\in U$ such that
$w \in (y,z_1)$.
As $w\in H$ and $y\not\in H$, we have $z_1\not\in T_1$.
However $U\subseteq T_1\cup \dots \cup T_k$ and hence $z_1\in T_i$ for some $i\ge 2$.
Then $[y,z_1]\subseteq T_i$ and we conclude that $\emptyset \ne (y,w)\subseteq T_1^\circ\cap T_i$ with $i>1$, a contradiction.
\end{proof}

The cell of $\te$ formed by the intersection of the tiles containing $x$ is the smallest cell containing $x$ and we call it the \emph{cell} of $\te$ \emph{generated} by $x$.
By Lemma~\ref{CellGeneratedL}, the relative interiors of the cells of $\te$ form a partition of $\XQ$. The cell generated by $x$ is the unique cell of $\te$ whose relative interior contains $x$.

\begin{lemma} \label{SidePropertiesL}
Let $\te$ be a locally finite partial tessellation of $\XQ$. Let $T_1$ and $T_2$ be two tiles of $\te$ and let $S=T_1 \cap T_2$.
Assume  $S$ has codimension $1$ and let $H$ be  the hyperplane generated by $S$.
Then the  following properties hold:
\begin{enumerate}
\item $H$ is an essential hyperplane of both $T_1$ and $T_2$.
\item $T_1$ and $T_2$ are contained in different closed half-spaces defined by $H$.
\item $S^r \cap T=\emptyset$ for every tile $T$ different from both $T_1$ and $T_2$.
\item $S$ is a side of $\te$.
 \end{enumerate}
\end{lemma}

\begin{proof}  Let $x\in S^r$ and let $U$ be an open neighborhood of $x$ such that $U \cap H \subseteq S$.
As $x\in \partial T_1$, from Lemma~\ref{EssentialFacesL} we obtain that $x\in H_1$, for some essential hyperplane $H_1$ of  $T_1$.
If $H_1\neq H$ then $H\cap U$ intersects non-trivially the two open half spaces defined by $H_1$, contradicting the fact that $T_1$ does not intersect one of these open half-spaces.
Therefore $H=H_1$. Hence (1)  follows.
It also proves that $H$ is the only essential hyperplane of $T_1$ containing $x$ and, by symmetry, it also is the only essential hyperplane of $T_2$ containing $x$.
So there is an open ball $B$ of $\XQ$ centered in $x$ and not intersecting any essential hyperplane of $T_1$ or $T_2$ different from $H$.
Let $Z$ be the closed half-space with boundary $H$ and containing $T_1$.
Then $B\cap Z^{\circ}$  is one of the two non-empty connected components of $B\setminus H$ and it is contained in $T_{1}^{\circ}$.
Since $T_{1}^{\circ} \cap T_2=\emptyset$ it follows that $T_2\subseteq Z'$. This proves (2).
The same argument then shows that $B\cap Z' \subseteq T_2$. Therefore $B\subseteq T_{1} \cup T_2$.
If $T$ is any tile such that $x\in T$ then $B$ contains a point in $T^{\circ}$. So $T^{\circ} \cap (T_1 \cup T_2) \neq \emptyset$ and therefore
$T=T_1$ or $T=T_2$. Hence (3) follows. Clearly (4) is a consequence of (3).
\end{proof}

The following proposition follows at once from Lemma~\ref{SidePropertiesL}.

\begin{proposition} \label{SideIncludedTwoTilesP}
 Every side of a locally finite partial  tessellation $\te$ is contained in exactly two tiles and it is the intersection of these  tiles.
 \end{proposition}

\begin{lemma}
\label{PartialTessellationL}
Let $\te$ be a locally finite partial tessellation of $\XQ$, let $H$ be a hyperplane of $\XQ$ and let $Z$ be a closed half-space defined by $H$.
Then
 $$\te_{Z} =\{ T\cap H : \;  T\in \te \text{ and }Z \text{ is an essential closed half-space of } T\}$$
is a locally finite partial tessellation of $H$.
\end{lemma}
\begin{proof}
Let  $T_1$ and $T_2$  be different elements of $\te_Z$.
If $(T_1\cap H)^r \cap (T_2\cap H)^r \neq \emptyset$ then $T_1\cap T_2$ is a side of $\te$   and $H$ is the subspace generated by this side.
Then, by  Lemma~\ref{SidePropertiesL}, $T_1$ and $T_2$ are in different closed half-spaces defined by $H$, a contradiction.
This proves that $\te_Z$ is a partial tessellation. As $\te$ is locally finite, so is $\te_Z$.
\end{proof}

\begin{proposition} \label{BoundaryUnionSides}
If $T$ is a tile of a locally finite tessellation $\te$ then $\partial T$ is the union of the sides of $T$ in $\te$.
\end{proposition}
\begin{proof}
Let $x\in \partial T$ and let $T_1, \ldots , T_k$ be the tiles containing $x$ and that are different from $T$.
Let $U$ be an open neighborhood of $x$ such that $T,T_1, \ldots , T_k$ are the only tiles intersecting  $U$.
By Lemma~\ref{EssentialFacesL}, there exists an essential hyperplane $H$  of $T$ in $\XQ$ such that $x\in H$.
Then $H\cap T$ is a thick  polyhedron of $H$.
Therefore, by Lemma~\ref{InteriorLemmaL}, $U\cap ( H\cap T)^{r}$ is a non-empty open subset of $H$.
Because $\te$ is a tessellation and $T\cap H \subseteq \partial T$, we get that $U\cap T \cap H \subseteq \cup_{i=1}^{k}
T\cap T_{i}$. So $U\cap ( H\cap T)^{r}$ is non-empty open subset of $H$ contained in  $\cup_{i=1}^{k}  T\cap T_{i}\cap H$.
Hence, $T\cap T_{i}\cap H$ is thick in $H$ for some $i$. So, by Lemma~\ref{SidePropertiesL}, $T\cap T_i$ is a side of $T$.
Therefore $x$ belongs to a side of $T$.
This proves one of the inclusions of the statement. The other one is obvious.
\end{proof}

\begin{proposition}\label{EdgesContainedTwoSidesP}
Let $T$ be a tile of a locally finite tessellation $\te$  of $\XQ$. If  $E$ is an edge of $T$ in $\te$ then $E$ is contained in exactly two sides of $T$ in $\te$.
\end{proposition}
\begin{proof} Let $E$ be an edge of $T$ in $\te$.
First we prove by contradiction that $E$ cannot be contained in three different sides of $T$ in $\te$.
So, assume that $E$ is contained in three distinct sides, say $S_{1}$, $S_{2}$ and  $S_3$ of $T$ in $\te$.
Because of Proposition~\ref{SideIncludedTwoTilesP}, $S_{i}=T\cap T_{i}$ with $T_i$ a tile different from $T$.
Let $H_i$ denote  the hyperplane generated by $S_i$.
Because of Lemma~\ref{SidePropertiesL}, each $H_i$ is an essential hyperplane of $T$.
Let $Z_i$ denote the closed half-space defined by $H_i$ such that $T\subseteq Z_i$.
By Lemma~\ref{SidePropertiesL}, $T_i \subseteq Z_i'$.
Furthermore, by Lemma~\ref{IntersectionTwoL},  we may assume that $Z_1\cap Z_2 \cap Z_3=Z_2\cap Z_3$.
As each $H_i$ is an essential hyperplane of $T$ we deduce that $H_1\in \{H_2,H_3\}$ and hence, we may assume that $H_1=H_2$.
Thus $Z_1=Z_2$.
If $H_3 =H_1$ then $Z_1=Z_3$ and $S_1 , S_2$ and $S_3$ are tiles of the partial tessellation $\te_{Z_{1}}$ of $H_{1}$ defined as in
Lemma~\ref{PartialTessellationL}.
Since  $E$ has codimension $1$ in $H_1$ and because it is contained in each $S_i$ we get that $S_i\cap S_j$ is a side of $\te_{Z_1}$ for every  $i\neq j$.
Consequently,  $S_1,S_2$ and $S_3$ are distinct tiles of $\te_{Z_1}$ containing points in the relative interior of $E$ and hence in the relative interior of a side of $\te_{Z_1}$.
This contradicts with Lemma~\ref{SidePropertiesL}.
So $H_3\neq H_1$ and thus  $H_1\cap H_3$ is the subspace generated by $E$. Moreover, $S_1 $ and $S_2$ are tiles of $\te_{Z_{1}}$ and
$S_1 \cap S_2$ is a side of $\te_{Z_{1}}$.
By Lemma~\ref{SidePropertiesL}, $S_1$ and $S_2$ are in different closed half-spaces of $H_1$ defined by
$H_1 \cap H_3$.
These closed half-spaces are $H_1 \cap Z_3$ and $H_1 \cap Z_3'$. By symmetry we also may assume that $S_2\subseteq Z_3'$.
Hence, $T\cap T_2 =S_2 \subseteq Z_1 \cap Z_3 \cap Z_1' \cap Z_3'\subseteq H_1 \cap H_3$,
in  contradiction with the fact that $S_2$ has codimension $1$.

It remains to prove that $E$ is contained in two different sides of $T$. Let $x\in E^r$.
Then $x\in \partial T$ and therefore $x\in S$ for some side $S$ of $T$ by Proposition~\ref{BoundaryUnionSides}.
Hence, by the definition of a cell, $E\subseteq S$. By Lemma~\ref{SidePropertiesL}, $S=T\cap T_1$ with $T_1$ a tile of $\te$ different from $T$.
Let $H$ be the hyperplane of $\XQ$ generated by $S$.
By Lemma~\ref{SidePropertiesL}, $H$ is an essential hyperplane of both $T$ and $T_1$.
Furthermore,   $T$ and $T_1$ are included in different closed half-spaces defined by $H$.
Note that, because of Lemma~\ref{EssentialFacesL},  a point $y\in S$ which is not in any essential hyperplane of $T$ or $T_1$ different from $H$ has a neighborhood contained in $T\cup T_1$.

We claim that there is a hyperplane $H_1$ of $\XQ$ different from $H$ such that $H_1$ is an essential hyperplane of either $T$ or $T_1$ and it intersects $E^r$ non-trivially. Indeed, for otherwise, for every $y\in E^r$ there is a neighborhood $U_y$ of $y$ in $\XQ$ such that $U_y \subseteq T\cup T_1$.
Then $\cup_{y\in E^r} U_{y}$ does not intersect any tile of $\te$ different from both $T$ and $T_1$.
In particular, the only tiles intersecting $E^r$ are $T$ and $T_1$.
So $T$ and $T_1$ are the only tiles containing $E$ and hence $E=S$, a contradiction. This proves the claim.

So let $H_1$ be a hyperplane different from $H$ such that $E^r\cap H_1\ne \emptyset$ and $H_1$ is essential hyperplane of either $T$ or $T_1$.
We claim that $E^r\subseteq H_1$. Otherwise $E$ has positive dimension and the subspace $L$ generated by $E$ is not contained in $H_1$.
Hence $H_1\cap L$ is a hyperplane of $L$, since $\emptyset\ne L\cap H_1\ne L $.
As $E^r$ is an open subset of $L$ of dimension at least 1, it has points in the two open half-spaces of $L$ defined by $H_1\cap L$.
This implies that $E^r$ has points in the two open half-spaces of $\XQ$ defined by $H_1$. This contradicts with the facts that $E^r\subseteq T\cap T_1$ and either $T$ or $T_1$ is contained in one closed half-space defined by $H_1$, because it is an essential hyperplane of either $T$ or $T_1$.
So we have proved that $E^r\subseteq H_1$.
Lemma~\ref{IntersectionTwoL} therefore yields that if $H_1$ is essential in $T$ (respectively, $T_1$) then $H$ and $H_1$ are the only essential hyperplanes of $T$ (respectively, $T_1$) intersecting $E^r$ non-trivially.

Let $x\in E^r$ and let $T,T_1,R_1,\dots,R_k$ be the different tiles of $\te$ containing $x$.
Then $E\subseteq T\cap T_1 \cap (\cap_{i=1}^k R_i )$, by the definition of a cell.

Assume that $H_1$ is an essential hyperplane of $T$.
Consider an open ball centered at $x$ such that $U\cap R=\emptyset$ for every tile $R\not\in \{T,T_1,R_1,\dots,R_k\}$ and $U\cap H_2=\emptyset$ for every essential hyperplane $H_2$ of $T$ with $H_2\not\in  \{ H,H_1\}$.
Let $V$ be the open half-space of $\XQ$ defined by $H$ and containing $T^{\circ}$.
Then $V\cap U\cap H_1$ is a non-empty open subset of $H_1$ contained in the boundary of $T$ and hence it is also contained in $\cup_{i=1}^k T\cap R_i$, because $V\cap T_1=\emptyset$.
Thus $T\cap R_i$ has codimension 1 for some $i$ and hence it is a side of $T$ containing $E$.

Finally assume that $H_1$ is an essential hyperplane of $T_1$ and not of $T$.
In this case we consider an open ball $U$ in $\XQ$ with center $x$ such that $U\cap R=\emptyset$ for every tile $R\not\in \{R_1,\dots,R_k\}$ and $U\cap H_2=\emptyset$ for every essential hyperplane $H_2$ of $T$ with $H_2\ne H$.
Moreover, let $Z_1$ be the open half-space defined by $H_1$ not containing $T_1$.
Then $Z_1\cap U \cap H$ is a non-empty open subset of $H$ contained in the boundary of $T$.
Hence $Z_1\cap U \cap H\subseteq \cup_{i=1}^k T\cap R_i$ and therefore $T\cap R_i$ is a side of $T$ containing $E$, for some $i$.

In both cases $E$ is contained in two different sides of $T$ containing $E$, namely $S$ and $T\cap R_i$, as desired.
\end{proof}

Note that if the dimension of $\XQ$ is at least 3 then an edge of a tile $T$ in $\te$ is not necessarily the intersection of two sides of $T$, although it is contained in exactly two distinct sides.
Moreover, even if an edge $E$ is the intersection of two sides of a tile, it could be properly contained in the intersection of two sides of another tile (see Figure~\ref{IntersectionSidesNotEdge}).

\begin{figure}
\begin{center}
\begin{tikzpicture}[scale=0.7]
\draw[-] (0,0) -- (6,0);
\draw[dashed] (-2,1) -- (4,1);
\draw[dashed] (-4,2) -- (2,2);
\draw[-] (0,0) -- (-4,2);
\draw[dashed] (3,0) -- (1,1);
\draw[dashed] (6,0) -- (2,2);

\draw[-] (0,3) -- (6,3);
\draw[red,line width=2.5pt] (-2,4) -- (1,4);
\draw[blue,line width=2.5pt,dashed] (1,4) -- (4,4);
\draw[dashed] (-4,5) -- (2,5);
\draw[-] (0,3) -- (-4,5);
\draw[dashed] (3,3) -- (1,4);
\draw[dashed] (6,3) -- (2,5);

\draw[-] (0,6) -- (6,6);
\draw[-] (-2,7) -- (4,7);
\draw[-] (-4,8) -- (2,8);
\draw[-] (0,6) -- (-4,8);
\draw[-] (6,6) -- (2,8);

\draw[-] (0,0) -- (0,6);
\draw[-] (3,0) -- (3,3);
\draw[-] (6,0) -- (6,6);

\draw[-] (-2,1) -- (-2,7);
\draw[dashed] (1,1) -- (1,4);
\draw[dashed] (4,1) -- (4,7);

\draw[-] (-4,2) -- (-4,8);
\draw[dashed] (2,2) -- (2,8);

\draw[-] (-2,4) -- (0,6);
\draw[dashed] (4,4) -- (6,6);

\draw (-1,1.5) node {$A$};
\draw (5,2) node {$B$};
\draw (-0.5,4.5) node {$C$};
\draw (-1.5,6) node {$D$};
\draw (-3,3) node {$E$};
\draw (-3,6) node {$F$};
\end{tikzpicture}
\caption{\label{IntersectionSidesNotEdge}
The intersection of the sides  $C\cap D$ and $D\cap F$ of $D$ is the union of the edges $A\cap D$ (red, fat and continuous) and $B\cap D$ (blue, fat and dashed).
None of these two edges is the intersection of two sides of $D$ while both are the intersection of two sides of $C$.}
\end{center}
\end{figure}
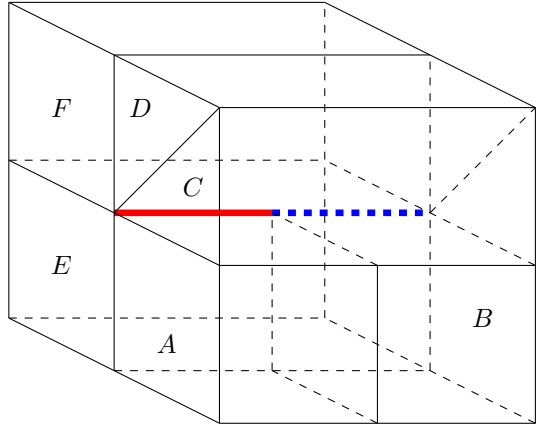

Let $E$, $S$ and $T$ be respectively an edge, a side and a tile of $\te$ with $E\subset S \subset T$.
We define recursively two sequences, one of tiles $(T_0,T_1,\dots)$ and another of sides $(S_1,S_2,\dots)$ by setting
\begin{align}\label{EdgeLoopD}
\begin{split}
T_0&=T,  \\
S_1&=S, \\
T_i &= \text{tile containing } S_i \text{ and different from } T_{i-1}, \text{ and}\\
S_{i+1} &= \text{side of } T_i  \text{ containing } E \text{ and different from } S_i. 
\end{split}
\end{align}
This is well defined by Proposition~\ref{SideIncludedTwoTilesP} and Proposition~\ref{EdgesContainedTwoSidesP} and we have for $i\ge 1$:
	\begin{center}
	 $T_{i-1}$ and $T_i$ are the only tiles containing $S_i$ and \\
	$S_i$ and $S_{i+1}$ are the only sides of $T_i$ containing $E$.
	\end{center}

For example, assume that Figure~\ref{IntersectionSidesNotEdge} represents part of a tessellation of $\R^3$ and take $T=A$, $S=A\cap C$ and $E=A\cap C$, the  red, fat, continuous  segment. Then, the sequence of tiles is periodic of period 5 starting with $(A,C,D,F,E)$.
If $T=B$, $S=B\cap C$ and $E$ is the blue, fat, dashed segment  then again the sequence of tiles is periodic of period 5 starting with $(B,C,D,F,E)$.
If one considers the edge $E=A\cap B \cap C$ then, with an appropriate side and tile selection, we obtain a sequence of tiles of period 3 starting with $(A,B,C)$.
We will show that this  behaviour is general.

As every edge is contained in finitely many tiles, the sequences only have finitely many different elements.
Moreover, if $k$ is minimum such that $T_k=T_m$ for some $m>k$ then $k=0$.
Indeed, $S_k$ and $S_{k+1}$ are the only sides of $T_k$ containing $E$ and the same happens for $S_m$ and $S_{m+1}$.
Therefore either $S_k=S_m$ and $S_{k+1}=S_{m+1}$ or $S_k=S_{m+1}$ and $S_{k+1}=S_m$.
In the former case $T_{m-1}$ contains $S_m=S_k=T_{k-1}\cap T_k$ and it is different from $T_m=T_k$.
Therefore, if $k\ne 0$ then $T_{m-1}=T_{k-1}$.
In the latter case, if $k\ne 0$ then $T_{m+1}$ contains $S_{m+1}=S_k=T_{k-1}\cap T_k$ and it is different form $T_m=T_k$.
Hence, $T_{m+1}=T_{k-1}$.
In both cases we obtain a contradiction with the minimality of $k$.

Let $m$ be the minimal positive integer with $T_0=T_m$, then
   \begin{eqnarray} \label{EdgeLoopEq}(T_0,T_1,\dots,T_m)&&
   \end{eqnarray}
is called an  edge loop of $E$.
This definition depends on the choice of the tile $T$ and the side $S$.
However, if $S$ is replaced by another side $S'$ containing $E$ and included in $T$ (there is only one option by Proposition~\ref{EdgesContainedTwoSidesP}) then the edge loop obtained is $(T_m,T_{m-1},\dots,T_2,T_1,T_0)$.
If we replace $T$ by one of the tiles $T_i$ and $S$ by one of the sided  of $T_i$ containing $E$, i.e. either $S_{i+1}$ or $S_i$ then the edge loop is either $(T_i,T_{i+1},\dots,T_m,T_1,\dots,T_{i-1},T_i)$ or $(T_i,T_{i-1},\dots,T_1,T_m,\dots,T_{i+1},T_i)$.
The next lemma shows that there are no other alternatives.

\begin{lemma}\label{EdgeLoopAllTilesL}
 If $E$ is an edge and $(T_0,T_1,\dots,T_m)$ is an edge loop of $E$ then $T_1\dots,T_m$ are precisely the tiles containing $E$.
\end{lemma}

\begin{proof}
Note that, as stated before, there are only finitely many tiles containing $E$.
We can order them in such a way that $T_0, \ldots , T_n$ are all these tiles and  $T_0, \ldots T_m$ are the tiles forming the edge loop.
Let $z \in E^r$. Then $z \in T_i$ for every $0 \leq i \leq n$ and there exists $\lambda >0$ such that the ball $B=B(z, \lambda)$ intersects a tile $T$ if and only $T=T_i$ for for some $0 \leq i \leq n$.
Moreover
$$B = B_1 \cup B_2$$
with
 $$B_1  =  \bigcup_{i=0}^m B \cap T_i \quad \text{ and } \quad B_2= \bigcup_{i=m+1}^n B \cap T_i .$$

We prove the result by contradiction. So suppose that $n>m$.
Hence, both $B_1$ and $B_2$ are non-empty closed sets.
Moreover $B_1 \cap B_2
\subseteq \bigcup_{0 \leq i \leq m, m+1 \leq j \leq n} \left(T_i \cap T_j\right)$.
We claim that $T_i \cap T_j$ is of codimension $1$ for at least one $0 \leq i \leq m$ and one $m+1 \leq j \leq n$.
Otherwise set $B_1' = B_1 \setminus \left(
B_1 \cap B_2 \right)$, $B_2' = B_2 \setminus \left(
B_1 \cap B_2 \right)$ and  $B'=B \setminus
\left(B_1 \cap B_2 \right)  = B_1' \cup B_2'$. Clearly, $B_1'$ and $B_2'$ are disjoint. Moreover, $T_i$ is thick and
$T_i^0$ is dense in $T_i$. Hence $B\cap T_i$ contains an open subset of $\XQ$.
Thus $\dim B_i=\dim \XQ$. However $\dim (T_i\cap T_j) < \dim \XQ$ for each $i\ne j$ and hence $\dim B_1\cap B_2<\dim \XQ$.
Thus $B'_i\ne \emptyset$.
Moreover, as $T_i \cap T_j$ has codimension at least $2$ for
every $ 0 \leq i \leq n$ and $m+1 \leq j \leq n$, by Lemma~\ref{ConnectedCod2L},  $B'$ is path-connected and hence
connected. Thus, $B'$ is
a connected subspace as the intersection of two disjoint closed subspaces, which is a
contradiction. Hence there exists $0 \leq i \leq m$ and $m+1 \leq j \leq n$ such that $T_{i} \cap T_{j}$ is of
codimension $1$, and hence it is a side containing $E$ by Proposition~\ref{SideIncludedTwoTilesP}.
Denote this side by $S^*$.
By the definition of an edge loop, $S_{i}= T_{i-1} \cap T_{i}$ and $S_{i+1} = T_{i} \cap T_{i+1}$ are two different sides
contained in $T_{i}$ and containing $E$ (indices are interpreted modulo $m$.) Moreover $T_{i-1} \neq T_{j}$ and $T_{i+1} \neq T_{j}$ and hence $E$ is
contained in three different sides, which contradicts Lemma~\ref{EdgesContainedTwoSidesP}.
\end{proof}

%
%
%

\section{Group presentations}
\label{SectionGroupPresentations}
\begin{definition}
Denote by $\Isom (\XQ)$ the group of isometries of $\XQ$.
A subgroup $G$ of $\Isom(\XQ)$ is said to be \emph{discontinuous} if for every compact subset $K$ of $\XQ$ there are only finitely many $g\in G$ with $g(K)\cap K \neq \emptyset$.
Using that $\XQ$ is separable it is easy to see that every discontinuous subgroup of $\Isom(\XQ)$ is countable.

Let  $G$ be a discontinuous group of isometries of $\XQ$.
A \emph{fundamental polyhedron} of  $G$ is a polyhedron $P$  of $\XQ$ such that $\te_P =\{ \gamma (P) \; : \; \gamma \in G\}$ is a tessellation of $\XQ$.
The polyhedron  $P$ is said to be  \emph{locally finite} if $\te_P$ is locally finite.
\end{definition}

Throughout this section $G$ is a discontinuous subgroup of the group of isometries of $\XQ$.
The action of $G$ on $\XQ$ induces a dimension preserving action on the set of subspaces of $\XQ$.
In particular, $G$ acts on the set of polyhedra of $\XQ$.

Observe that if $P$ is a fundamental polyhedron of $G$ then
\begin{enumerate}
\item $\XQ=\cup_{\gamma\in G} \gamma(P)$ and
\item $P^\circ\cap \gamma(P)^\circ=\emptyset$ for every $1\ne \gamma\in G$.
\end{enumerate}
Conversely, assume that $P$ satisfies (1) and (2). Then $\alpha(P)^\circ\cap \beta(P)^\circ=\emptyset$ for any  distinct  elements $\alpha$ and $\beta$ of $G$.
Moreover, if $P$ were not thick in $\XQ$ then it follows from Lemma~\ref{InteriorLemmaL} that it is contained in a hyperplane. As $G$ is countable, we get that   $\cup_{\gamma \in G} \gamma (P)$ is contained in a countable union of hyperplanes, in contradiction with Lemma~\ref{ConnectedCod2L}.
Hence $P$ is thick and hence $\gamma(P)$ is thick for every $\gamma\in G$.
Thus $P$ is a fundamental polyhedron of $G$ if and only if (1) and (2) hold.

In this section, we show in Theorem~\ref{PoincareT} that a presentation of a group may be established based on the tessellation given by its fundamental domain. The main idea to get a generating set is the following. Given a non-trivial element $g \in G$, one considers a path from a point inside the fundamental domain $P$ of $G$ to a point inside $g(P)$. The path can be chosen such that it intersects only intersections of images of $P$ of codimension $1$. Every such intersection corresponds to an element of $G$ called a \textit{side-pairing transformation} and it may then be shown that $g$ can be written as a product of those side-pairing transformations. Theorem~\ref{PoincareT} also gives the relations between the different generators. We will give more details about this later.

Throughout this section
 \begin{center}
  $P$ is a locally finite fundamental polyhedron of $G$ and $\te=\te_P$.
 \end{center}
When we refer to cells, tiles, sides or edges it is always with respect to $\te$.
Since every cell is contained in only finitely many tiles, the stabilizer of one cell is finite.
If $S$ is a side of $P$ then, by Proposition~\ref{SideIncludedTwoTilesP}, there is a unique $g\in G\setminus \{1\}$ such that $S=P\cap g(P)$ and $S^r \cap h(P)=\emptyset$ for every $h\in G\setminus \{1,g\}$.
We denote this $g$ as $\gamma_S$.
It is called a \emph{(side) pairing transformation}.
So,
  $$S=P\cap \gamma_{S}(P).$$
If $g$ is a pairing transformation then $P\cap g^{-1}(P)=g^{-1}(P\cap g(P))$ also is a side of $P$ and hence $g\inv$ is a pairing transformation as well.
In this case one denotes $S_g=P\cap g(P)$ and one says then that $S_g$ and $S_{g^{-1}}$ are paired sides.
If $S$ is a side then the side paired with $S$ is denoted $S'$. So,
    $$S'=\gamma_S\inv(S).$$
If $E$ is an edge of $P$ then, by Proposition~\ref{EdgesContainedTwoSidesP},  it is contained in exactly two sides, say $S_g$ and $S_{g_1}$.
Because $G$ permutes edges,  $g^{-1}(E)$ is an edge of the tessellation  and it is contained in $S_{g\inv}$.
Therefore $g\inv(E)$ and $g_1\inv(E)$ are edges of $P$.

Some relations amongst the side pairing transformations can be deduced.
A first type of relations is easily obtained. Indeed, if   $S_1$ and $S_2$ are two paired sides  then $\gamma_{S_1}=\gamma_{S_2}\inv$.
Such a  relation is called a {\it pairing relation}.
In case $S$ is a side paired with itself then the pairing relation takes the form $\gamma_{S}^{2}=1$ and such a relation usually is called a {\it reflection relation}.

To define the second type of relations, we introduce the following definition.

\begin{definition}
A \emph{loop} of $G$ with respect to $P$ (or simply a loop of $G$, if the polyhedron is clear from the context) is a finite ordered list $( g_{0}, g_{1}, \; \ldots ,\; g_{n})$ of elements of $G$ such that $g_0  = g_{n} $ and
$g_{i-1}(P) \cap g_{i}(P)$ is a side for each $1\leq i \leq n$ (equivalently, each $g_{i-1}^{\inv} g_i $ is a  pairing transformation).
\end{definition}

\begin{example}\label{LoopsEx}{\rm
\begin{enumerate}
\item If $g,h\in G$ and $g\inv h$ is a pairing transformation then $(g,h,g)$ is a loop of $G$.
\item If $(T_0=g_0(P),T_1=g_1(P),\dots,T_m=g_m(P))$ is an edge loop then $(g_0,g_1,\dots,g_m)$ is a loop of $G$.
\item If $\gamma_{S_1}\cdots \gamma_{S_m}=1$ for sides $S_1,\dots,S_n$ of $P$ then
	$$(1,\gamma_{S_1},\gamma_{S_1}\gamma_{S_2},\dots,\gamma_{S_1}\cdots \gamma_{S_{n-1}},\gamma_{S_1}\cdots \gamma_{S_m}=1)$$
is a loop.
Conversely, if $(g_0,g_1,\dots,g_m)$ is a loop of $G$ then  $S_i=P\cap g_{i-1}\inv g_i(P)$ is a side of $P$ and $g_{i-1}\inv g_i=\gamma_{S_i}$ for every $i\in \{ 1,\dots,m\}$. Furthermore, $\gamma_{S_1}\cdots \gamma_{S_m}=1$ .
\end{enumerate}
}\end{example}

By Example~\ref{LoopsEx}.(3), relations amongst pairing transformations are completely determined by loops.
In case a loop $(g_0, \ldots , g_m)$ is determined by an edge loop, as in Example~\ref{LoopsEx}.(2), then the resulting relation $\gamma_{S_{1}} \cdots \gamma_{S_{m}} =1$ is called an  \emph{edge loop relation}.
Note that the pairing relations are the relations corresponding to the loops $(g,h,g)$ for $g\inv h$ a pairing transformation.

We give an alternative interpretation of the edge loop relations.
Let $E$ be an edge of $P$ and choose one of the two sides $S$ of $P$ containing $E$.
Recursively one obtains a sequence $E_1,E_2,\dots$ of edges of $P$ and a sequence of sides $S_1,S_2,\dots$, with $E_i\subseteq S_i$ for each $i$, and  which is uniquely determined by the following rules:
\begin{align}\label{EdgeD}
\begin{split}
&E_1=E, \quad S_1=S, \quad E_{n+1}=\gamma_{S_n}\inv(E_n) \quad \text{and} \\
&	S_{n+1} \text{ and } S'_n=\gamma_{S_{n}}^{-1}(S_n) \text{ are the two sides of } P \text{ containing } E_{n+1}.
\end{split}
\end{align}
Let $g_n=\gamma_{S_1}\cdots \gamma_{S_n}$ for every $n\ge 0$ (in particular, we agree that $g_0=1$).
Observe  that $g_n\inv(E)=E_{n+1}\subseteq P$.
In particular $E\subseteq g_{n}(P)$ and hence $E$ is contained in the tiles
	$$T_0=g_0(P)=P, T_1=g_{1}(P), T_2=g_{2}(P), \dots$$
and $T_{n-1}\cap T_n$ is a side of $\te$ for every $n\ge 1$.
Moreover $g_n\inv(T_{n-1}\cap T_n)=S'_n\ne S_{n+1}=g_n\inv(T_n\cap T_{n+1})$. Therefore $T_{n-1}\cap T_n$ and $T_n\cap T_{n+1}$ are the two sides of $T_n$ containing $E$.
This proves that $T_0=P,T_1,T_2,\dots$ is a  sequence of tiles as defined in (\ref{EdgeLoopD}).
We know this  is a periodic sequence and if it has period $m$ then $(T_0,T_1,\dots,T_m)$ is the edge loop defined by $E$,$S$ and $P$ and $(g_0=1,g_1,\dots,g_m)$ is the loop of $G$ associated to this edge loop.
As $T_i$ determines $g_i$,  the sequence $g_0,g_1,\dots$ also is periodic of period $m$.
As $E_i=g_i\inv(E)$ and $S_i=P\cap g_{i-1}\inv g_i(P)$, the sequence  of pairs $(E_{i},S_{i})$ also is periodic,
say of period  $k$ and let $t=\frac{m}{k}$.
Then $t$ is a positive integer and the edge loop relation associated to the loop $(g_0,g_1,\dots,g_m)$ takes the form $1=g_m=(\gamma_{S_1}\cdots \gamma_{S_k})^t$.
This  usually is called a cycle relation. This is the second type of relations we need for the Poincar\'{e} result.
Observe that cycle relation  and edge loop relation are synonymous concepts.

Since $m$ is the minimum integer so that $g_m=g_0=1$ and $g_{k t}=(\gamma_{S_1}\cdots \gamma_{S_k})^{t}$,
we deduce that $t$ is the order of $\gamma_{S_1}\cdots \gamma_{S_k}$.
An alternative way to see that $\gamma_{S_1}\dots \gamma_{S_k}$ has finite order is by observing that it stabilizes the edge $E$ and the stabilizer of every cell is finite.

Some of the cycle relations are redundant. For example, if $S$ and $R$ are the two sides of $P$ containing the edge $E$ and if  $(T_0=P,T_1,T_2,\dots,T_{m-1},T_m=P)$ is the edge loop obtained by applying the above procedure to $E$ and $S$ then the edge loop obtained by applying the procedure to $E$ and $R$ is  $(T_0=P,T_{m-1},\dots,T_2,T_1,T_0=P)$. These two loops  give rise to equivalent cycle relations: $(\gamma_{S_1}\dots \gamma_{S_k})^t=1$ and $(\gamma_{S_k}\inv \dots \gamma_{S_1}\inv)^t=1$.
This is because if the period of the list $(E_i,S_i)$ obtained from $E$ and $S$ is $k$ then $\gamma_{S_{k-1}}\inv(S_{k-1})'=R$ and hence the list of pairs of edges and sides starting with $(E,R)$ is $(E,R),(E_{k-1},S_{k-1}'),\dots,(E_1,S_1'),\dots$.
On the other hand, if we replace $E$ by one of the edges $E_i$ then the sequence of pairs of edges and sides obtained is a shift of the list obtained with $E$ and $S$ or $R$. Then, the cycle relation obtained with $E_i$ is a conjugate of the cycle relation associated with $E$.

The edges in the list $E_1,\dots,E_k$ form a \emph{cycle of edges} of $P$. Clearly, the non-equivalent cycles of edges of $P$ define a partition of the edges of $P$.

\begin{example}{\rm
Let $n\ge 3$ and let $D_{2n}$ be the group of isometries of a regular polygon of the Euclidean plane with $n$ sides.
Then, the acute wedge $P$ between the two half-lines $S_1=\{(x,0):x>0\}$ and $S_2=\left\{\left(x,x\tan(\frac{\pi}{n}\right):x>0\right\}$ is a fundamental polyhedron of $D_{2n}$.
Let $g_i$ be the reflection in the line containing $S_i$. Then $S_i=P\cap g_i(P)$ and $g_i^2=1$.
So $S_1$ and $S_2$ are the two sides of $P$, as fundamental polyhedron of $G$ and the pairing relations are the reflection relations $\gamma_{S_1}^2=\gamma_{S_2}^2=1$.
The only edge is the vertex  consisting of the single  point $(0,0)$. The sequence of edges and sides starting with $E$ and $S_1$ is periodic of period 2. Clearly,   $\gamma_{S_1}\gamma_{S_2}$ has finite order and in fact it has order $n$ because it is the rotation around $(0,0)$ of angle $\frac{2\pi}{n}$. Therefore the only cycle relation is $(\gamma_{S_1}\gamma_{S_2})^n=1$.}
\end{example}

We are ready to state Poincar\'{e}'s Theorem on presentations of discontinuous groups.

\begin{theorem}[Poincar\'{e}] \label{PoincareT}
Let $\XQ$ be either an Euclidean, hyperbolic or  spherical space.
Let $P$ be a locally finite fundamental polyhedron  for a discontinuous group of isometries on the space $\XQ$.
The  pairing transformations generate $G$, that is
  $$G = \GEN{\gamma \in G \; : \;  P \cap \gamma(P) \textrm{ is a side of } P},$$
and the pairing and cycle relations form a complete set of relations for $G$.
\end{theorem}

We first  prove  that the pairing transformations generate $G$.
The proof that the pairings and cycle relations form a complete set of relations requires much more work and will be postponed until the end of the section.

\vspace{12pt}
\begin{proof} {\bf of Theorem~\ref{PoincareT} (Generators).}
Let $\te$ be the tessellation of $\XQ$ formed by the polyhedra $g(P)$ with $g\in G$.
As $G$ is countable, $\te$ is countable and, as every non-empty cell is the intersection of finitely many tiles (those containing a relative interior point of the cell), the number of cells of $\te$ is countable.
Let $Y$ be the complement in $\XQ$ of the union of the cells of codimension at least 2.
Then, by Lemma~\ref{ConnectedCod2L},  $Y$ is path-connected.

Let $g\in G$ and let  $x\in P^{\circ}$. So $g(x)\in g(P)^{\circ}$ and $x,g(x)\in Y$.
Because $Y$ is path connected, there exists a continuous function $\alpha : [0,1]\rightarrow Y$ with $\alpha (0)=x$ and $\alpha (1)=g(x)$.
The path $\alpha$ can be very odd and we need to choose a path that ``travels smoothly'' through the tiles of $\te$. Therefore,  we need to perform some ``deformations'' on $\alpha$. To do so,  we let $\mathcal{P}$ denote the finite set consisting of all the tiles that intersect the compact set $\alpha ([0,1])$ and define
    $$L=\left\{l\ge 1 :
    \matriz{{l}
    \text{There is a continuous function } \beta :[0,1] \rightarrow Y, \\
    \text{a sequence } 0=a_0<a_1<\dots <a_l\le 1 \text{ and different tiles } T_1,\dots,T_l \\
    \text{such that } \beta (a_0)=x, \beta ((a_l,1]) \cap (T_{1}\cup \cdots T_{l})=\emptyset, \alpha\mid_{(a_l,1]}=\beta\mid_{(a_1,1]}, \\
   \beta (a_i)=\alpha (a_i), \text{ and } \beta ([a_{i-1},a_i]) \subseteq T_i\text{ for every } 1\le i\le l
   }
    \right\}.$$

  We first prove that $1\in L$. To show this,  put $a_{1}=\max (\alpha^{-1}(P))$, $T_1=P$
  and define the continuous function $\beta :[0,1] \rightarrow Y$  as follows.
The restriction of $\beta$ to $[0,a_1 ]$ is such that its image runs through the geodesic $[x,\alpha (a_1)]$ from $x$ to
$\alpha (a_1)$.
The restriction of $\beta$ to $[a_1,1]$ is the same function as the restriction of $\alpha$ to $[a_1,1]$.
Clearly $\beta$, $0<a_0< a_1 \le 1$ and $T_1$ satisfy the conditions of the definition of $L$.
Observe that if $l\in L$ then $l\le |\mathcal{P}|$ and hence $L$ is bounded. Let $l$ be maximal in $L$ and let $\beta$, $0=a_0<a_1<\dots<a_l\le 1$ and $T_1,\dots,T_l$ satisfy the conditions of the definition of $L$.
We claim that  $a_l=1$.
Indeed,  if $a_l<1$ we get that $\beta (a_l)\in \partial T_l$ and thus, by Lemma~\ref{BoundaryUnionSides},  $\beta (a_l)\in S$, for some side $S$ of $T_{l}$.
Since $\beta (a_l) \in Y$ (and thus $\beta (a_l)$ is not in a cell of codimension at least $2$) we get that   $S$ is the cell of $\te$ generated by $\beta (a_l)$. Hence, $\beta (a_l) \in S^r$.
By Proposition~\ref{SideIncludedTwoTilesP}, $S=T_{l}\cap T$ for some tile $T$ of $\te$ with $T\neq T_l$.
Furthermore, $T$ and $T_l$ are the only tiles of $\te$ containing $S$.
So, $\beta (a_l)$ does not belong to any other tile of $\te$.
Then, $T_l\cup T$ contains a neighborhood of $\beta (a_l)=\alpha (a_l)$.
Because $\beta ((a_l,1)) \cap (T_1\cup \cdots \cup T_l) = \alpha ((a_l ,1]) \cap (T_{1}\cup \cdots T_l) =\emptyset$, it follows that
   $\beta ((a_l, b))\subseteq T$
for some $a_l<b\leq 1$.
Hence $T\neq T_i$ for every $i\in \{ 1, \ldots , l\}$.
Let $a_{l+1} =\max (\alpha^{-1}(T))$.
So, $a_{l+1}>a_l$.
Put $T_{l+1}=T$.
One can now define a continuous function
  $\beta ' :[0,1] \rightarrow Y$
as follows.
On $[a_0,a_l] \cup (a_{l+1},1]$ the function $\beta '$ agrees with $\beta$.
On  $[a_l,a_{l+1} ]$ the function $\beta '$  is such that its image runs through the geodesic $[\beta (a_l), \beta (a_{l+1})]$ from $ \beta (a_l)$ to
$\beta (a_{l+1})$ in case  $\beta (a_l)$  and $\beta(a_{l+1})$ are not antipodal and otherwise  one chooses $x\in T^{\circ}$ and then one defines on $[a_l,a_{l+1} ]$ the function $\beta '$   such that its image runs first through the geodesic $[\beta (a_l),x]$ and then through the geodesic $[x, \beta(a_{l+1})]$.
So $\beta ' ([a_l , a_{l+1}]) \subseteq T_{l+1}$.
Clearly, $\beta ' ([a_{i-1},a_{i}]) =\beta ([a_{i-1},a_{i}])\subseteq T_{i}$ for $1\leq i\leq l$ and
$\beta ' ((a_{l+1},1]) \cap (T_{1}\cup \cdots T_{l+1}) =\emptyset$.
So $\beta'$, $a_0<a_1<\dots<a_l<a_{l+1}\le 1$ and $T_1,\dots,T_l,T_{l+1}$ satisfy the conditions of the definition of $L$ and  this contradicts with the maximality of $l$.

So, indeed $a_l=1$ and thus there exists a continuous function $\alpha : [0,1]\rightarrow Y$ and a sequence
$0=a_{0} < a_1 < \ldots < a_{l} =1$ such that $\alpha (a_0)=x$, $\alpha (a_l)=g(x)$ and
$\alpha ([a_{i-1},a_{i}]) \subseteq T_{i}$ with $T_{i}$ tile of $\te$ for each $1\leq i\leq l$.
Clearly $T_0=P$.
Write $T_{i}=g_i(P)$ with $g_i\in G$.
Then,
$g_1=1$, $g_l=g$ and $\alpha (a_i)\in g_i(P)\cap g_{i+1}(P)$  for every $i\in \{ 1,2,\dots,l-1\}$.
Since $\alpha(a_i)$ is not in any cell of codimension greater than 1, each $g_i(P)\cap g_{i+1}(P)$ is a side of $g_i(P)$ and therefore $P\cap g_i\inv g_{i+1}(P)$ is a side of $P$.
Hence $\gamma_i=g_i\inv g_{i+1}$ is a pairing transformation for each $1\le i<l$.
Finally $g=g_l=(g_1\inv g_2)(g_2\inv g_3)\cdots (g_{l-1}\inv g_l)=\gamma_1\cdots \gamma_{l-1}$.
Hence, $g$ belongs to the subgroup of $G$ generated by the pairing transformations.
Because $g$ is an arbitrary element of $G$, the result follows.
\end{proof}

\begin{remark}
Note that the previous proof can be simplified by using \cite{alexandrov}. Indeed at the end of the proof of \cite[Theorem 1]{alexandrov}, the author proves that ``every segment in $\R^n$ can be covered by a sequence of polyhedra in which every two consecutive polyhedra are adjacent at an $(n-1)$-dimensional face''. Using this, one can then easily construct a sequence
$0=a_{0} < a_1 < \ldots < a_{l} =1$ such that $\alpha (a_0)=x$, $\alpha (a_l)=g(x)$,
$\alpha ([a_{i-1},a_{i}]) \subseteq T_{i}$ with $T_{i}$ tile of $\te$ and such that $T_i \cap T_{i+1}$ is a side of $T_i$ for each $1\leq i\leq l$. The proof then finishes in the same way as above. However, since our goal is to make the proof self-contained, and avoid ambiguities, the proof given above seemed best suited.
\end{remark}

Let $\Delta$  denote the group given by the presentation of Theorem~\ref{PoincareT}. More precisely,
  $$\Delta=F/N,$$
where $F$ is the free group with basis  the symbols $[g]$, one for each pairing transformation $g$, and $N$ is the normal closure of the subgroup of $F$ generated by the set $X$ consisting  of the  pairing and cycle relations, i.e. $X$ is formed by the products $[\gamma_S][\gamma_{S'}]$ with $S$ and $S'$ paired sides of $P$ (pairing relations) and the elements of the form $([\gamma_{S_1}]\dots [\gamma_{S_k}])^t$, where $(E_1=E,S_1=S,E_2,S_2,\dots)$ is the list defined by (\ref{EdgeD}) for an edge $E$ and a side $S$ containing $E$, $k$ is the period of the list and
$t$ is the order of $\gamma_{S_1}\dots \gamma_{S_k}$ (cycle relations).

It is clear that the function 
\begin{eqnarray*} \varphi:  \Delta & \rightarrow & G \\
\left[g \right] & \mapsto & g
\end{eqnarray*}
is surjective. So, in order to prove the relation part of Theorem~\ref{PoincareT}, we just have to prove the injectivity of $\varphi$, i.e. we have to prove that if $g_1g_2 \ldots g_n=1$ for some $g_i \in G$, then $\left[g_1\right]\left[g_2\right]\ldots \left[g_n\right]=1$. The main idea therefore is to link products in $\Delta$ with loops in the space on which $G$ is acting. We then show that if $g_1g_2 \ldots g_n=1$, then the loop associated to $\left[g_1\right]\left[g_2\right]\ldots \left[g_n\right]$ is homotopic to the trivial loop, i.e. a point, and hence its value is $1$.

Recall that the edge loop relations and the cycle relations are synonymous concepts, so we may replace the cycle relations by the edge loop relations
$$[g_0\inv g_1][g_1\inv g_2]\cdots [g_{n-1}\inv g_n] ,$$ where $(g_0,g_1,\dots,g_{n-1},g_n)$ is an edge loop. Abusing notation, we will consider the symbols $[g]$, with $g$ a pairing transformation, as elements of $\Delta$. Hence,
	$$[\gamma_S][\gamma_{S'}]=1, \text{ for every side } S \text{ of } P$$
(so $[\gamma_{S}]^{-1}=[\gamma_{S}^{-1}]$)
and
	$$[g_0\inv g_1][g_1\inv g_2]\cdots [g_{n-1}\inv g_n]=1, \text{ for every edge loop } (g_0,g_1,\dots,g_n) \text{ of } P.$$

Let $g,h\in G$ and let $C$ be a cell of $\te$ of codimension $m$ at most $2$ that is  contained in $g(P)\cap h(P)$.
We define $\kappa_C (g,h)\in \Delta$ as follows.
\begin{itemize}
\item If $m=0$ then $\kappa_C(g,h)=1$.
\item If $m=1$ then $\kappa_C (g,h) =[g\inv h]$.
\item If $m=2$ then $C$ is an edge contained in $g(P)\cap h(P)$ and thus, by Lemma~\ref{EdgeLoopAllTilesL}, $g$ and $h$ belong to the edge loop of $C$. Up to a cyclic permutation, we can write the edge loop of $C$ as  $(g= k_0, \ldots , k_t =h, k_{t+1}, \ldots , k_m=g)$ (or the equivalent edge loop $(g=k_m,k_{m-1},\dots,k_t=h,k_{t-1},\dots,k_1,k_0=g)$) and we set
	$$\kappa_C (g,h) =[k_0\inv k_1][k_1\inv k_2]\;  \cdots [k_{t-1}\inv k_t] = [k_m\inv k_{m-1}] \cdots [k_{t+1}\inv k_{t}].$$
\end{itemize}

Observe that $\kappa_{C}(g,g)=1$ in the three cases.

\begin{lemma}\label{kappaElementaryL}
Let $g,h\in G$ and let $C$ be a cell of $\te$ of codimension $m\le 2$ that is contained in $g(P)\cap h(P)$. The  following properties hold.
\begin{enumerate}
\item\label{kappaAntisymmetric} $\kappa_C(g,h)=\kappa_C(h,g)\inv$.
\item\label{kappaIndependence} If $D$ is cell of $\te$ contained in $C$ and of codimension at most $2$ then $\kappa_D(g,h)=\kappa_C(g,h)$.
\item\label{kappaTransitiveCell} If $g_1,\dots,g_n\in G$ and $C\subseteq \bigcap_{i=1}^n g_i(P)$ then
    $\kappa_C(g_1,g_n)=\kappa_C(g_1,g_2)\cdots \kappa_C(g_{n-1},g_n)$.
\end{enumerate}
\end{lemma}

\begin{proof}
(\ref{kappaAntisymmetric})
If $m=0$, then $g=h$ and there is nothing to prove.
If $m=1$ then  $g^{-1}h =\gamma_{S_{g^{-1}h}}$ and $h^{-1}g= \gamma_{S_{h^{-1}g}} =\gamma_{S_{g^{-1}h}'}$ and hence $\kappa_C(g,h)\kappa_C(h,g)=[\gamma_{S_{g\inv h}}][\gamma_{S'_{g\inv h}}]=1$, a pairing relation. Finally if $m=2$, then we can write the edge loop of $C$ as $(g= k_0, \ldots , k_t =h, k_{t+1}, \ldots , k_m=g)$ and thus
\begin{eqnarray*}
\kappa_C (g,h) &  = & [k_0\inv k_1][k_1\inv k_2]\;  \cdots [k_{t-1}\inv k_t], \\
\kappa_C (h,g) & = &[k_t\inv k_{t+1}][k_{t+1}\inv k_{t+2}]\;  \cdots [k_{m-1}\inv k_m].
\end{eqnarray*}
It is now easy to see that $\kappa_C(g,h)\kappa_C(h,g)=1$ and hence the result follows.

(\ref{kappaIndependence}) If $C=D$ then there is nothing to prove. So assume that $C\ne D$. If $C$ is a side then $D$ is an edge and $g$ and $h$ are two consecutive elements of the edge loop of $D$. Then $\kappa_D(g,h)=[g\inv h]=\kappa_C(g,h)$. Otherwise, $C$ is a tile and hence $g=h$. Thus $\tau_D(g,h)=1=\tau_C(g,h)$.

(\ref{kappaTransitiveCell}) By induction it is enough to prove the statement for $n=3$. So assume $n=3$. If either $g_1=g_2$ or $g_2=g_3$ then the desired equality is obvious. So assume that $g_1\ne g_2$ and $g_2\ne g_3$.
If $C$ is an edge then, up to a cyclic permutation, possibly  reversing the order and making use of
Lemma~\ref{EdgeLoopAllTilesL}, the edge loop of $C$ is of the form $(g_1=k_0,\dots,g_2=k_t,\dots,g_3=k_l,\dots,k_m)$. Then,
\begin{eqnarray*}
    \kappa_C(g_1,g_3)&=&[k_0\inv k_1][k_1\inv k_2]\cdots [k_{l-1}\inv k_l] \\
   &=& ([k_0\inv k_1][k_1\inv k_2]\cdots [k_{t-1}\inv k_t]) \; ( [k_t\inv k_{t+1}] \cdots [k_{l-1}\inv k_l]) \\
   &=& \kappa_C(g_1,g_2)\kappa_C(g_2,g_3)
\end{eqnarray*}
Otherwise, $S=g_1\inv(C)$ is a side of $P$, $\gamma_S=g_1\inv g_2$, $\gamma_{S'}=g_2\inv g_1$ and $g_1=g_3$.
Then,
	$$\kappa_C(g_1,g_3) = 1 = [\gamma_S][\gamma_{S'}] = \kappa_C(g_1,g_2)\kappa_C(g_2,g_3).$$
\end{proof}

We denote by $\XQ_{\te}$ the complement in $\XQ$ of the union of the cells of $\te$ of codimension at least three.
By Lemma~\ref{CellGeneratedL}, every element of $\XQ_{\te}$ is either in the interior of a tile or in the relative interior of a side or an edge.
The first ones are those that belong to exactly one tile, the elements of the relative interior of one side belong to exactly two tiles (Lemma~\ref{SidePropertiesL})  and the remaining elements belong to at least three tiles.

If $x\in \XQ_{\te}$, $g,h\in G$ and $x\in C\subseteq g(P)\cap h(P)$ for some cell $C$ then the codimension of $C$ is at most $2$ and we define $$\kappa_x(g,h)=\kappa_C(g,h).$$
This is well defined because if $D$ is another cell containing $x$ and contained in $g(P)\cap h(P)$ with $\kappa_C(g,h)\ne \kappa_D(g,h)$ then $g\ne h$ and $C\ne D$.
Hence neither $C$ nor  $D$ is a  tile and either $C$ or $D$ is a side. Therefore, $g(P)\cap h(P)$ is a side and hence, by Lemma~\ref{kappaElementaryL}.(3),  $\kappa_C(g,h)=[g\inv h]=\kappa_D(g,h)$, a contradiction.
This proves that indeed $\kappa_{x}(g,h)$ is well defined.
By Lemma~\ref{kappaElementaryL} we have $\kappa_x(g,h)=\kappa_x(h,g)\inv$ and if $x\in \cap_{i=1}^n g_i(P)$ with $g_1,\dots,g_n\in G$ then
	\begin{equation}\label{kappaTransitive}
	 \kappa_x(g_1,g_n)=\kappa_x(g_1,g_2)\cdots \kappa_x(g_{n-1},g_n)
	\end{equation}

\begin{lemma}\label{PhiOneStepL}
Let $\alpha:[0,1]\rightarrow \XQ_{\te}$ be a continuous function and let $0\le a < b < c \le 1$ and $g,h,k\in G$ be such that $\alpha(a)\in g(P)$, $\alpha((a,c))\subseteq C^r$ for a cell $C$ of $k(P)$ and $\alpha ((a,b))\subseteq D^r$ for a cell $D$ of $h(P)$.
Then
	$$\kappa_{\alpha(a)}(g,k) = \kappa_{\alpha(a)}(g,h)\; \kappa_{\alpha(b)}(h,k).$$
\end{lemma}
\begin{proof}
First of all, observe that $\alpha(a),\alpha(b)\in C\cap D$, $\alpha (a) \in g(P)\cap k(P)\cap h(P)$, $\alpha (b)\in h(P)\cap k(P)$, $C \subseteq h(P)$  and $D \subseteq k(P)$, because $\alpha$ is continuous and every cell is the closure of its relative interior.
The desired equality is clear if $h=k$.
So, assume that $h\ne k$.
Then $C$ and $D$ are not tiles because they are included in $h(P)\cap k(P)$.
If $C\subseteq g(P)$ then $\kappa_{\alpha(a)}(g,k) = \kappa_C(g,k) = \kappa_C(g,h)\; \kappa_C(h,k) = \kappa_{\alpha(a)}(g,h)\; \kappa_{\alpha(b)}(h,k)$, by Lemma~\ref{kappaElementaryL}.(\ref{kappaTransitiveCell}).
Assume that $C\not\subseteq g(P)$. In particular $g,h$ and $k$ are pairwise different and $\alpha(a)\not\in C^r$.
Hence $C$ is a side, because Lemma~\ref{CellGeneratedL} implies that  $\XQ_{\te}\cap E=E^r$ for every edge $E$.
Again by Lemma~\ref{CellGeneratedL}, we obtain that $\alpha(a)$ belongs to the relative interior of a cell $E$ of $g(P)$ and $E$ is properly contained in $C$.
Therefore $E$ is an edge and $h$ and $k$ appear\sout{s} consecutively in an edge loop of $E$, i.e. after a cyclic permutation or a reversing in the ordering, an edge loop of $E$ takes the form $(k_1=g,\dots,k_i=h,k_{i+1}=k,\dots,k_m)$. Then,
$\kappa_{\alpha(a)}(g,k)=\kappa_E(g,k)=([k_1\inv k_2]\cdots [k_{i-1}\inv k_i])\; [k_i\inv k_{i+1}] = \kappa_{E}(g,h)\kappa_C(h,k)=\kappa_{\alpha(a)}(g,h)\kappa_{\alpha(b)}(h,k)$, as desired.
\end{proof}

\begin{definition}
Let $\alpha : [a,b] \rightarrow \XQ_{\te}$ be a continuous function on a compact interval
$[a,b]$.
An $\alpha$-\emph{adapted list} is an  ordered list $\mathcal{L}=(a_0, g_1, a_1, g_2, \ldots , g_n, a_n)$ such that $a=a_0 < a_1 < \ldots < a_n=b$ and $\alpha (a_{i-1},a_i)$ is contained in the relative interior of a cell of $g_{i}(P)$ for all $1\leq i \leq n$.

Given an $\alpha$-adapted list $\mathcal{L}$, we define
	$$\Phi (\mathcal{L}) = \kappa_{\alpha(a_1)}(g_1, g_2) \; \kappa_{\alpha(a_2)} (g_2,g_3) \cdots \kappa_{\alpha(a_{n-1})}( g_{n-1},g_n),$$
unless $n=1$, where we set $\Phi(\mathcal{L})=1$.
\end{definition}

Observe that if $i\in \{1,\dots,n\}$ then $\alpha(a_i)$ belongs to the boundaries of both $\alpha((a_{i-1},a_i))$ and $\alpha((a_i,a_{i+1}))$. Hence, $\alpha(a_i)\in g_i(P)\cap g_{i+1}(P)$ and thus  $\kappa_{\alpha(a_i)}(g_i,g_{i+1})$ is well defined.

\begin{lemma} \label{IndependencyAdaptedL}
Let $\alpha : [a,b]  \rightarrow \XQ_{\te}$ be a continuous function such that both $\alpha(a)$ and $\alpha(b)$ belong to the interior of some tile.
If $\mathcal{L}$ and $\mathcal{L}'$ are  $\alpha$-adapted lists then $\Phi (\mathcal{L})=\Phi (\mathcal{L}')$.
\end{lemma}
\begin{proof}
Without loss of generality, we may assume that $[a,b]=[0,1]$.
Let $\mathcal{L}=(a_0,g_1,a_1,\dots,g_n,a_n)$ and $\mathcal{L}'=(a'_0,g'_1,a'_1,\dots,g'_m,a'_m)$.
First observe that $g_1(P)=g_1'(P)$ is the only tile containing $\alpha(0)$. Thus $g_1=g'_1$.
If $n=1$ then $\alpha([0,1))\subseteq g_1(P)^\circ$ and therefore $g'_i=g_1$ for every $i\in \{1,\dots,m\}$. Thus $\Phi(\mathcal{L})=1=\Phi(\mathcal{L}')$.
Similarly, if $m=1$ then $\Phi(\mathcal{L})=1=\Phi(\mathcal{L}')$. In the remainder of the proof we assume that $n,m>1$.

We construct an $\alpha$-adapted  list $\mathcal{D}_{\alpha}'$ containing $\mathcal{L}'$ and all the $a_i$'s as follows.
For every $1\le i\le i'<n$ and $j\in \{1,\dots,m-1\}$ such that $a_{i-1} \le a'_{j-1}< a_i \le a_{i'} < a'_j \le a_{i'+1}$, we insert in $\mathcal{L}'$ the sublist $(g_i,a_i,\dots,g_{i'},a_{i'})$ between $a'_{j-1}$ and $g'_j$.
Similarly, we construct another list $\mathcal{D}_{\alpha}$ containing $\mathcal{L}$ and all the $a'_i$'s.
We can consider the transition from $\mathcal{L}'$ to $\mathcal{D}_{\alpha}'$ (or from $\mathcal{L}$ to $\mathcal{D}_{\alpha}$) as the result of inserting finitely many  pairs $(g_i,a_i)$.
On the other hand, we can consider the transition from $\mathcal{D}_{\alpha}$ to $\mathcal{D}_{\alpha}'$ as the result of replacing finitely many  group elements.
Therefore, to prove the lemma, it is enough to deal with the following two cases:
(1) $\mathcal{L}'$ is obtained by inserting in $\mathcal{L}$ one pair $(h,b)$ between $a_{i-1}$ and $g_i$ for $a_{i-1}<b<a_{i}$ and $h\in G$; (2) $\mathcal{L}'$ is obtained by replacing in $\mathcal{L}$ one $g_i$ by $h$, and in both cases $\alpha((a_{i-1},a_i))$ is contained in the relative interior of a cell contained in $g_i(P)\cap h(P)$ and $h\ne g_i$.

(1) Assume  $\mathcal{L}'$ is obtained by inserting  one pair $(h,b)$ between $a_{i-1}$ and $g_i$.
We consider separately the cases when $i>1$ or $i=1$.
If $i>1$ then $a=a_{i-1}$, $b$, $c=a_i$, $g=g_{i-1}$, $h$ and $k=g_i$ satisfy the hypothesis of Lemma~\ref{PhiOneStepL} and therefore
	\begin{eqnarray*}
	 \Phi(\mathcal{L})&=&\kappa_{\alpha(a_1)}(g_1,g_2)\cdots \kappa_{\alpha(a_{n-1})}(g_{n-1},g_n) \\
	&=&
	\kappa_{\alpha(1)}(g_1,g_2)\cdots \kappa_{\alpha(a_{i-2})}(g_{i-2},g_{i-1}) \kappa_{\alpha(a_{i-1})}(g_{i-1},h) \kappa_{\alpha(b)}(h,g_i)  \\ & & \kappa_{\alpha(a_i)}(g_{i},g_{i+1}) \dots \kappa_{\alpha(a_{n-1})}(g_{n-1},g_n) \\
	&=& \Phi(\mathcal{L}').
	\end{eqnarray*}
If $i=1$ then $h=g_1$ because $\alpha(a_0)$ is in the relative interior of a tile.
Thus $\kappa_{\alpha(b)}(h,g_1)=1$ and hence
	\begin{eqnarray*}
	 \Phi(\mathcal{L})&=&\kappa_{\alpha(a_1)}(g_1,g_2)\cdots \kappa_{\alpha(a_{n-1})}(g_{n-1},g_n) \\
	&=&\kappa_{\alpha(b)}(h,g_1)\kappa_{\alpha(a_1)}(g_1,g_2)\cdots \kappa_{\alpha(a_{n-1})}(g_{n-1},g_n) \\
	&=& \Phi(\mathcal{L}').
	\end{eqnarray*}

(2) Assume $\mathcal{L}'$ is obtained by replacing $g_i$ with  $h\ne g_i$ in $\mathcal{L}$.
By definition, $\alpha((a_{i-1},a_i))\subseteq E^r$ for some cell $E$ contained in $g_i(P)\cap h(P)$.
Since $g_i\ne h$, clearly  $E$ is either an edge or a side.
As both $\alpha(0)$ and $\alpha(1)$ belong to the interior of tiles we have $i\ne 1,n$.
To prove $\Phi(\mathcal{L})=\Phi(\mathcal{L}')$, it is enough to show
$$\kappa_{\alpha(a_{i-1})}(g_{i-1},h)\kappa_{\alpha(a_i)}(h,g_{i+1})=\kappa_{\alpha(a_{i-1})}(g_{i-1},g_i)\kappa_{\alpha(a_i)}(g_i,g_{i+1}).$$
Let $D_1$ and $D_2$ be the cells generated by $\alpha(a_{i-1})$ and $\alpha(a_i)$ respectively.
Then $D_1,D_2\subseteq E$.
If $E$ is an edge then $D_1=D_2=E$ and hence
    \begin{eqnarray*}
    \kappa_{\alpha(a_{i-1})}(g_{i-1},h)\kappa_{\alpha(a_i)}(h,g_{i+1})&=&
    \kappa_{E}(g_{i-1},h)\kappa_{E}(h,g_{i+1})\\
    & =&  \kappa_{E}(g_{i-1},g_{i+1}) \\
    &=& \kappa_{E}(g_{i-1},g_i)\kappa_{E}(g_i,g_{i+1}) \\
    &=& \kappa_{\alpha(a_{i-1})}(g_{i-1},g_i)\kappa_{\alpha(a_i)}(g_i,g_{i+1}),
    \end{eqnarray*}
by Lemma~\ref{kappaElementaryL}.(\ref{kappaTransitiveCell}).
Otherwise, $E$ is a side.
If $D_1$ is a side then $D_1=E$ and $g_{i-1}\in \{g_i,h\}$. Thus
    \begin{eqnarray*}
    \kappa_{\alpha(a_{i-1})}(g_{i-1},h)\kappa_{\alpha(a_i)}(h,g_{i+1})&=&
    \kappa_{E}(g_{i-1},h)\kappa_{D_2}(h,g_{i+1})\\
    & =&  \kappa_{D_2}(g_{i-1},h)\kappa_{D_2}(h,g_{i+1}) \\
    &=& \kappa_{D_2}(g_{i-1},g_{i+1}) \\
    &= & \kappa_{D_2}(g_{i-1},g_i)\kappa_{D_2}(g_i,g_{i+1}) \\
    &=& \kappa_{E}(g_{i-1},g_i)\kappa_{D_2}(g_i,g_{i+1}) \\
    &=& \kappa_{\alpha(a_{i-1})}(g_{i-1},g_i)\kappa_{\alpha(a_i)}(g_i,g_{i+1}),
    \end{eqnarray*}
by statements (\ref{kappaIndependence}) and (\ref{kappaTransitiveCell}) of Lemma~\ref{kappaElementaryL}.
The case where $E=D_2$ is proved similarly.
Finally, assume that $E$ is a side and $D_1\ne E$ and $E\ne D_2$.
Then $D_1$ and $D_2$ are edges and $g_i$ and $h$ are consecutive members of the edge loops of $D_1$ and $D_2$.
After some reordering the edge loop of $D_1$ takes the form $(k_1=g_{i-1},k_2,\dots,k_t=g_i,k_{t+1}=h,\dots,k_m)$ and the edge loop of $D_2$ takes the form $(h_1=g_i,h_2=h,\dots,h_l=g_{i+1},\dots,h_n)$. Then $[k_t\inv k_{t+1}]=[g_i\inv h]=[h_1\inv h_2]$ and
\begin{eqnarray*}
\kappa_{\alpha(a_{i-1})}(g_{i-1},h)\; \kappa_{\alpha(a_i)}(h,g_{i+1})&=&
([k_1\inv k_2] \dots [k_{t-1}\inv k_t][k_t\inv k_{t+1}])\;  ([h_2\inv h_3]\dots [h_{l-1}\inv h_l]) \\
&=&([k_1\inv k_2] \dots [k_{t-1}\inv k_t]) \; ([h_1\inv h_2] [h_2\inv h_3]\dots [h_{l-1}\inv h_l] )\\
&=&\kappa_{\alpha(a_{i-1})}(g_{i-1},g_i)\; \kappa_{\alpha(a_i)}(g_i,g_{i+1}).
\end{eqnarray*}
This finishes the proof.
\end{proof}

To finish the proof of Theorem~\ref{PoincareT}, Lemma~\ref{PhiConstantL} will be crucial. To prove this lemma in a smooth way, we need to construct piecewise geodesic simple paths. Therefore we introduce the definition in the following paragraph along with Lemma~\ref{AdaptedExistenceL} and Lemma~\ref{StrongHomotopicL}. As suggested by the referee, this argumentation could be simplified by applying typical generic position/transversality arguments. However for completeness sake we prefer to give a detailed self-contained proof of such a construction.

A \emph{parametrization of a geodesic segment} $[x,y]$ in $\XQ$ is a surjective continuous function $\alpha:[a,b]\rightarrow [x,y]$ with $\alpha(a)=x$, $\alpha(b)=y$ and such that the map $t\rightarrow d(x,\alpha(t))$ is not decreasing (recall that $d$ denotes the distance function on $\XQ$). For such a map, if $T$ is a polyhedron then
$\alpha\inv(T)$ is a closed interval (maybe empty or of 0 length). Indeed, it is closed because so is $P$ and $\alpha$ is continuous. To prove that $\alpha\inv(T)$ is an interval, let $a\le t_0 < t_1 \le b$, with $\alpha(t_0),\alpha(t_1)\in T$ then $\alpha([t_0,t_1])$ is the geodesic segment $[\alpha(t_0),\alpha(t_1)]$
and hence it is contained in $T$ because $T$ is convex.
Therefore $[t_0,t_1]\subseteq \alpha\inv(T)$. This proves that $\alpha\inv(T)$ indeed is an interval.

By
   $$\mathcal{C}([a,b],\XQ_{\te})$$
we denote the set consisting of the continuous functions $\alpha:[a,b]\rightarrow \XQ_{\te}$ for which there is a finite ascending sequence $a=b_0< b_1 \dots < b_m=b$ such that, for every $i\in \{1,\dots,m\}$, the restriction of $\alpha$ to $[b_{i-1},b_i]$ is a parametrization of a geodesic segment.

\begin{lemma}\label{AdaptedExistenceL}
 If $\alpha\in \mathcal{C}([a,b],\XQ_{\te})$ then there is an $\alpha$-adapted list.
\end{lemma}
\begin{proof}
Without loss of generality, we may assume that $[a,b]=[0,1]$.
By restricting to the intervals $[b_{i-1},b_i]$, one may assume without loss of generality that $\alpha$ is a parametrization of a geodesic segment.
Then, for every cell $C$, $\alpha\inv(C)$ is a closed interval of $[0,1]$ (maybe empty or of length $0$) and the image of $\alpha$ intersects finitely many tiles.
We claim that for every $t\in [0,1)$ there is a cell $C$ and an $\epsilon>0$ such that $\alpha((t,t+\epsilon))\subseteq C^r$ and
for every $t\in (0,1]$ there is a cell $D$ and an $\epsilon>0$ such that $\alpha((t-\epsilon ,t))\subseteq D^r$.
By symmetry, we only prove the first statement.
So,   fix $t\in [0,1)$ and assume that  $T_1,\dots,T_k$ are the only tiles containing $\alpha(t)$.
For every $i\in \{ 1,\dots,k\}$, let $\epsilon_i=\max \{\epsilon \geq 0 : \epsilon \leq 1-t, \; \alpha(t+\epsilon)\in T_i\}$.
Let $U$ be a neighborhood of $\alpha(t)$ not intersecting any tile different from every $T_1, \ldots , T_k$. Then there is an
$\epsilon>0$ such that $\alpha([t,t+\epsilon])\subseteq U$ and therefore $\alpha(t+\epsilon)\in T_i$ for some $i$. Hence, $\epsilon_i>0$ for some $i\in \{1,\dots,k\}$.
Renumbering if necessary, we may assume that $m$ is a positive integer such that $m\leq k$ and
$\epsilon_i>0$ if and only if $i\le m$. Let $\epsilon=\min \{\epsilon_1,\dots,\epsilon_m\}$.
Then $\alpha((t,t+\epsilon))\subseteq T_1\cap \dots \cap T_m$ and $T_1,\dots,T_m$ are the unique tiles containing an element of
$\alpha ((t,t+\epsilon))$. This implies that $T_1\cap \dots \cap T_m$ is a cell, say $C$, and $\alpha((t,t+\epsilon))$ is contained in the relative interior of $C$ by Lemma~\ref{CellGeneratedL}.
This proves the claim.

Let $c$ be the supremum of all  $t \in [0,1]$ with the property that  there exists an $\alpha_t$-adapted list where $\alpha_t$ denotes the  restriction of  $\alpha $ to  $[0,t]$. By the  claim $c>0$.
It remains to be shown that $c=1$.
Otherwise, by the claim, there exists $\epsilon >0$ such that $\alpha ((c-\epsilon,c))$ is contained in the relative interior of a cell $C$ and  $\alpha ((c,c+\epsilon))$ is contained in the relative interior of a cell $D$. Let $h,k\in G$ be such that $C\subseteq h(P)$ and $D\subseteq k(P)$.
If $(a_0=0, g_1, a_1, \dots, g_n,a_n=c-\epsilon)$ is an $\alpha_{c-\epsilon} $-adapted list then   $(a_0=0, g_1, a_1, \dots, g_n,c-\epsilon, h,c,k,c+\epsilon)$ is an $\alpha_{c+\epsilon}$-adapted list, contradicting the maximality of $c$.
\end{proof}

Let $x,y \in \XQ$ and
let
   $$\mathcal{C}_{x,y}([a,b],\XQ_{\te})=\{\alpha\in \mathcal{C}([a,b],\XQ_{\te}): \alpha (a) =x, \alpha (b)=y.\}$$
On $\mathcal{C}_{x,y}([a,b],\XQ_{\te}) $ we consider the metric $d$ defined as follows. If
$\alpha,\beta\in \mathcal{C}([a,b],\XQ_{\te})$ then $d(\alpha,\beta)=\max(d(\alpha(c),\beta(c)):c\in [a,b]\}$.
Assume also that both $x$ and $y$ belong to the interior of some tile of $\te$.
By Lemma~\ref{IndependencyAdaptedL} and Lemma~\ref{AdaptedExistenceL} there is a well defined map
$$\Phi :  \Gamma  \rightarrow \Delta $$
given by
 $$\Phi (\alpha ) =\Phi (\mathcal{L}),$$
for $\mathcal{L}$ an $\alpha$-adapted list.
The next aim is proving that the map $\Phi:\mathcal{C}_{x,y}([0,1],\XQ_{\te})\rightarrow \Delta$ is constant.
To do so,  we first prove a strong simply connected property on $\XQ_{\te}$ with respect to the elements of $\mathcal{C}_{x,y}([a,b],\XQ_{\te}) $.
Recall that if $\alpha,\beta\in \mathcal{C}_{x,y}([a,b],\XQ_{\te})$ then a homotopy from $\alpha$ to $\beta$ is a continuous function $H:[a,b]\times [a,b]\rightarrow \XQ_{\te}$ such that $H(a,t)=\alpha(t)$, $H(b,t)=\beta(t)$, $H(t,0)=x$ and $H(t,1)=y$ for every $t\in [0,1]$.
We say that $\alpha$ and $\beta$ are \emph{strongly homotopic} if there is an homotopy from $\alpha$ to $\beta$ such that $H(t,-)\in \mathcal{C}_{x,y}([a,b],\XQ_{\te})$ for every $t\in [a,b]$.
Clearly, this defines an equivalence relation on $\mathcal{C}_{x,y}([a,b],\XQ_{\te})$.
Moreover, this equivalence relation is preserved by concatenation.
More precisely, if  $\alpha\in \mathcal{C}_{x,y}([a,b],\XQ_{\te})$ and $\beta\in \mathcal{C}_{y,z}([a,b],\XQ_{\te})$ then the concatenation of $\alpha$ and $\beta$ is  the function $\alpha\oplus \beta:[a,b]\rightarrow \XQ_{\te}$ defined by
   $$(\alpha\oplus \beta)(t)=
        \left\{
            \matriz{{ll} \alpha (x +2(t-a)(\frac{y-x}{b-a})), & \text{if } a\le t \le \frac{a+b}{2}; \\
                              \beta (y+2(t-\frac{a+b}{2})(\frac{y-z}{b-a}) ), & \text{if } \frac{a+b}{2} \le t \le b .}
        \right.
$$
If $\alpha_i,\beta_i \in \mathcal{C}_{x,y}([a,b],\XQ_{\te})$ are so that $\alpha_i$ and $\beta_i$ are strongly homotopic for $i\in \{1,2\}$ then $\alpha_1\oplus \beta_1$ and $\alpha_2\oplus \beta_2$ are strongly homotopic.

\begin{lemma}\label{StrongHomotopicL}
All the elements of $\mathcal{C}_{x,y}([a,b],\XQ_{\te})$ are strongly homotopic.
\end{lemma}

\begin{proof}
Let $V$ and $W$ be two subspaces of $\XQ$ of dimension $m$ and $n$ respectively and let $\GEN{V,W}$ denote the smallest subspace of $\XQ$ containing $V$ and $W$.
If $V\cap W\ne\emptyset$ then the dimension of $\GEN{V,W}$ is at most $m+n$. Otherwise the dimension of $\GEN{V,W}$ is at most $m+n+1$.
Indeed, let $w\in W$ and let $V_1$ be the smallest subspace of $\XQ$ containing $V$ and $w$. Then $V_1$ has dimension $m+1 $, $\GEN{V,W}=\GEN{V_1,W}$ and the dimension of this space is at most $m+n+1$, because $V_1\cap W\ne\emptyset$.

Let $L$ be a geodesic line of $\XQ$ and let $V$ be a geodesic subspace of $\XQ$ of codimension at least 3. Then, by the previous,  $\GEN{L,V}$ has positive codimension.
Therefore, if $\{L_i:i\in I\}$ is a countable family of geodesic lines and $\{V_j:j\in J\}$ is a countable family of subspaces of codimension at least 3 then, by Lemma~\ref{ConnectedCod2L},  $\cup_{i\in I,j\in J} \GEN{L_i,V_j}$  is a proper subset of $\XQ$.
Using this for the case when $\{V_j:j\in J\}$ is the family of subspaces generated by the cells of tiles of $\te$ of codimension at least 3, we deduce that for every countable family $S=\{[x_i,y_i]:i\in I\}$ of geodesic segments contained in $\XQ_{\te}$, there exists  $w\in \XQ_{\te}\setminus \cup_{i\in I,j\in J} \GEN{L_i,V_j}$, where $L_i$ denotes the geodesic line containing
$[x_i,y_i]$.
This implies that the intersection of $\GEN{L_i,V_j}$ with the geodesic plane containing both  $L_i$ and $w$ is contained in $L_i$.
Thus, the cone of $[x_i,y_i]$ with vertex $w$ does not intersect any $V_j$ and hence it is contained in $\XQ_{\te}$.
(Observe that if $\XQ$ is spherical then the antipode of $w$ is not in any $[x_i,y_i]$ because $w\not\in L_i$.)

Let $x',y'\in \XQ$ with $x'\ne y'$ and let $\rho=d(x',y')$.
Let $\alpha_0 :[0,1]\rightarrow [x',y']$ be the parametrization of a segment $[x',y']$  of constant speed, that is $\alpha_0$ is the inverse of the map $z\in [x',y'] \mapsto \frac{d(x',z)}{\rho}\in [0,1]$.
Assume $\alpha$ is an arbitrary parametrization of the segment $[x',y']$.
Consider the function
$H:[0,1]\times [0,1]\rightarrow [x',y']$ defined  by
	$$H(s,t)=\alpha_0\left(\frac{d(x',\alpha(t))}{\rho}+s\left(t-\frac{d(x',\alpha(t))}{\rho}\right)\right).$$
As both the distance function $d$ and $\alpha$ are continuous, $H$ is a continuous function.
Thus $H(0,t)=\alpha(t)$ and $H(1,t)=\alpha_0(t)$.
Furthermore $H(s,0)=\alpha_0(0)= x'$ and $H(s,1)=\alpha_0(1)=y'$.
Moreover, as $\alpha$ is a parametrization, the function $t\mapsto d(x',\alpha(t))$ is non-decreasing.
Therefore, for every $s\in [0,1]$, the function
$t\mapsto \frac{d(x',\alpha(t))}{\rho}+s\left(t-\frac{d(x',\alpha(t))}{\rho}\right)=st+\frac{(1-s)d(x',\alpha(t))}{\rho}$ is non-decreasing.
Hence
	$$t\mapsto d(x',H(s,t))=d\left(x',\alpha_0\left(\frac{d(x',\alpha(t))}{\rho}+s\left(t-\frac{d(x',\alpha(t))}{\rho}\right)\right)\right)$$
is non-decreasing too. Therefore $H(s,-)$ is a parametrization of $[x',y']$, for every $s\in [0,1]$.
As the image of $H$ is $[x',y']$, we have showed that   $H$ is a strong homotopy between $\alpha$ and $\alpha_0$.

Without loss of generality, we may assume that $[a,b]=[0,1]$.
Let $\alpha,\beta\in \mathcal{C}_{x,y}([0,1],\XQ_{\te})$. We need to show that $\alpha$ and $\beta$ are strongly homotopic.
Clearly, there is an ascending finite list $0=a_0<a_1<\dots a_n=1$ such that the restrictions to $[a_{i-1},a_i]$ of $\alpha$ and  $\beta$ are both parametrizations of segments.
By the previous paragraph, we may assume without loss of generality that the restriction to each segment $[a_{i-1},a_i]$ of $\alpha$ and $\beta$ has constant speed.
We now argue by induction on $n$. If $n=1$ then $\alpha=\beta$ (because of the constant speed) and hence there is nothing to prove.
Let $x_i=\alpha(a_i)$ and $y_i=\beta(a_i)$ for $i\in \{ 0,1,\dots,n\}$.
By the discussion in the second paragraph of the proof,  there exists  $w\in \XQ_{\te}$ such that all the cones determined by the segments $[x_{i-1},x_i]$ and $[y_{i-1},y_i]$ and centered in $w$ are contained in $\XQ_{\te}$.
Let $\alpha_i$ denote the restriction of $\alpha$ to $[a_{i-1},a_i]$ and let $\beta_i$ denote the restriction of $\beta$ to $[a_{i-1},a_i]$.
Further, let $\alpha':[0,1]\rightarrow \XQ_{\te}$ be such that it agrees with $\alpha$ on $[0,a_{n-2}]$, the restriction of $\alpha'$ to $[a_{n-2},a_{n-1}]$ is a parametrization of the interval $[x_{n-2},w]$ of constant speed and the restriction of $\alpha'$ to $[a_{n-1},a_n]$ is a parametrization of $[w,y]$ of constant speed.
Similarly, let $\beta':[0,1]\rightarrow \XQ_{\te}$ agree with $\beta$ on $[0,a_{n-2}]$, its restriction to  $[a_{n-2},a_{n-1}]$ is a parametrization of the geodesic interval $[y_{n-2},w]$ of constant speed and the restriction of $\beta'$ to $[a_{n-1},1]$ is a parametrization of $[w,y]$ of constant speed.
By the induction hypothesis, the restrictions of $\alpha'$ and $\beta'$ to $[0,a_{n-1}]$ are strongly homotopic.
Furthermore $\alpha'$ and $\beta'$ coincide on $[a_{n-1},1]$ and hence $\alpha'$ and $\beta'$ are strongly homotopic.
It remains to prove that $\alpha$ and $\alpha'$ are strongly homotopic, and that so are $\beta$ and $\beta'$.
For this it is enough to prove that the restrictions of $\alpha$ and $\alpha'$ (respectively, $\beta$ and $\beta'$) to $[a_{n-2},a_n]$ are strongly homotopic.
This reduces the problem to the case where $n=2$ and the two geodesic triangles $x\alpha(a_1)\beta(a_1)$ and $\alpha(a_1)\beta(a_1)y$ are contained in $\XQ_{\te}$.
Let $\gamma:[0,1]\rightarrow [\alpha(a_1),\beta(a_1)]$ be a parametrization of $[\alpha(a_1),\beta(a_1)]$ of constant speed.
For every $s\in [0,a_1]$ let $t\rightarrow H(s,t)$ be the parametrization of $[x,\gamma(s)]$ of constant speed and, for $s\in [a_1,1]$, let $t\rightarrow H(s,t)$ be the parametrization of $[\gamma(s),y]$ of constant speed.
In other words, if $0\le t \le a_1$ then $H(s,t)$ belongs to the geodesic segment $[x,\gamma(s)]$ and
  \begin{eqnarray}\label{HomContEq1}
  d(x,H(s,t))&=&d(x,\gamma(s))\frac{t}{a_1} .
  \end{eqnarray}
On the other hand, if $a_1\le t \le 1$ then $H(s,t)$  belongs to the geodesic segment $[\gamma(s),b]$ and
  \begin{eqnarray}\label{HomContEq2}
  d(\gamma(s),H(s,t))&=&d(\gamma(s),y)\frac{t-a_1}{1-a_1} .
  \end{eqnarray}
Clearly $H(s,-)\in \mathcal{C}_{x,y}([0,1],\XQ_{\te})$ for every $s\in [0,1]$, $H(0,-)=\alpha$ and $H(1,-)=\beta$.
Finally, it is easy to see that the function $H$ is continuous.
\end{proof}

\begin{lemma}\label{PhiConstantL}
If both $x$ and $y$ belong to the interior of some tile then  $\Phi:\mathcal{C}_{x,y}([a,b],\XQ_{\te}) \rightarrow \Delta$ is a constant mapping.
\end{lemma}

\begin{proof}
Again, without loss of generality, we may assume that $[a,b]=[0,1]$.
We claim that it is sufficient to show that $\Phi$ is locally constant.
Indeed, assume this is the case and let  $\alpha,\beta\in \mathcal{C}_{x,y}([0,1],\XQ_{\te})$.
By Lemma~\ref{StrongHomotopicL} there is a strong homotopy $H$ from $\alpha$ to $\beta$.
Let $c$ denote  the supremum of the $s\in [0,1]$ for which $\Phi(H(s,-))=\Phi(\alpha)$.
Since, by assumption,  $\Phi$ is constant in a neighborhood of $H(x,-)$, it easily follows that $c=1$ and thus $\Phi(\alpha)=\Phi(\beta)$.

To prove that $\Phi$ is locally constant, we show that for every $\alpha\in \mathcal{C}_{x,y}([0,1],\XQ_{\te})$ and for every  $\alpha$-adapted list $\mathcal{L}=(a_0,g_1,a_1,\dots,g_n,a_n)$ (which exists because of
Lemma~\ref{AdaptedExistenceL}), there is a positive real number $\delta$ such that for every
$\beta\in \mathcal{C}_{x,y}([0,1],\XQ_{\te})$ with $d(\alpha,\beta)<\delta$ there is a $\beta$-adapted list
$\mathcal{D}=(b_0,h_1,b_1,\dots,h_m,b_m)$ and an increasing sequence of integers $j_0=0<j_1<j_2<\dots < j_{n-1}<j_n=m$ such that
   $\mathcal{L}'=(a_0,h_{j_1},a_1,h_{j_2},a_2,\dots,h_{j_{i-1}},a_{i-1},h_{j_n},a_n)$
is an $\alpha$-adapted list and, for every $i\in \{ 0,1,\ldots , n\}$,
   $\Phi((\mathcal{L}')_i)=\Phi((\mathcal{D})_i)$,
where
     $(\mathcal{L}')_i=(a_0,h_1,a_1,\dots,h_i,a_i)$
and
     $(\mathcal{D})_i=(b_0,h_1,b_1,\dots,h_{j_i},b_{j_i})$.
In particular, by Lemma~\ref{IndependencyAdaptedL}, we have
    $\Phi(\alpha)=\Phi(\mathcal{L})=
      \Phi(\mathcal{L}')=\Phi((\mathcal{L}')_n)=
       \Phi((\mathcal{D})_n)=\Phi(\mathcal{D})=\Phi(\beta)$,
as desired.

Since $\te$ is locally finite, there is $\delta_1>0$ such that for every $i\in \{1,\dots,n-1\}$ and every $g\in G$, if $B(\alpha(a_i),2\delta_1))\cap g(P)\ne \emptyset$ then $\alpha(a_i)\in g(P)$.
Since $\alpha$ is continuous there is $\epsilon<\min \left\{ \frac{a_i-a_{i-1}}{2}: i\in\{1,\dots,n\} \right\}$ such that, for every $i\in \{ 0,1,\ldots ,  n\}$, $d(\alpha(t),\alpha(a_i))<\delta_1$ for every $t$ with $|t-a_i|<\epsilon$.
For every $i\in \{1,\dots,n-1\}$, let $a'_i=a_i-\epsilon$ and $a''_i=a_i+\epsilon$. We also set $a'_n=1$ and $a''_0=0$.
Observe that $a''_{i-1}\le a'_i$ for every $i\in \{1,\dots,n\}$.
Each $\alpha([a''_i,a'_{i+1}])$ is compact and it is contained in the relative interior of a cell $C_i$ contained in $g_i(P)$.
Using again that $\te$ is locally finite we obtain a positive number $\delta_2$ such that $d(\alpha(t),g(P))>\delta_2$ for every $t\in [a''_i,a'_{i+1}]$ and every $g\in G$ with $C_i\not\subseteq g(P)$.
Let $\delta=\min\{\delta_1,\delta_2\}$.
We will prove that $\delta$ satisfies the desired property.
Let $\beta\in \mathcal{C}_{x,y}([0,1],\XQ_{\te})$ with $d(\alpha,\beta)<\delta$.
Then $d(\alpha(t),\beta(t))<\delta$ for every $t\in [0,1]$.
In particular,
\begin{equation}\label{Neighbour}
 \text{if } t\in (a'_i,a''_i) \text{ and } \beta(t)\in g(P) \text{ then } \alpha(a_i)\in g(P)
\end{equation}
because $d(\beta(t),\alpha(a_i))< 2\delta_1$.
Moreover,
\begin{equation}\label{Closed}
 \text{if } t\in [a''_{i-1},a'_i] \text { and } \beta(t)\in g(P) \text{ then }C_i \subseteq g(P),
\end{equation}
since $d(\alpha(t),\beta(t))<\delta_2$.
Furthermore
\begin{equation}\label{Interior}
 \text{if } C_i \text{ is a tile then }\beta([a''_{i-1},a'_i])  \subseteq C_i^\circ=g_i(P)^\circ.
\end{equation}
Indeed, as $\alpha([a''_{i-1},a'_i])\subseteq \alpha((a_{i-1},a_i))\subseteq C_i^\circ$, it follows that $C_i=g_i(P)$ is the only tile intersecting $\alpha([a''_{i-1},a'_i])$ and therefore it also is the only tile intersecting $\beta([a''_{i-1},a'_i])$. Then (\ref{Interior}) follows.

Let $\mathcal{D}=(b_0,h_1,b_1,\dots,h_m,g_m)$ be a $\beta$-adapted list.
We enlarge $\mathcal{D}$ by inserting each $a'_i$ and $a''_i$. More precisely, if we rename the list $(a'_1,a''_1,a'_2,a''_2,\dots,a'_{n-1},a''_{n-1})=(c_1,c_2,\dots,c_{2(n-1)})$, then we insert in $\mathcal{D}$ the sequence $(h_j,c_i,h_j,c_{i+1},\dots,h_j,c_k)$ between $b_{j-1}$ and $h_j$ for every $i$ whenever
$c_{i-1}\le b_{j-1} < c_i< \dots <c_k <b_j \le c_{k+1}$.
So we may assume without loss of generality that there is an ascending sequence $0=j'_0<j_1<j'_1<j_2<j'_2<\dots<j_{n-1}<j'_{n-1}<j_n=m$ such that
for every $i\in \{1,\dots,n\}$ $a'_i=b_{j_i}$ and for every $i\in \{0,1,\dots,n-1\}$ we have $a''_i=b_{j'_i}$.
We claim that $\mathcal{L}'=(a_0,h_{j_1},a_1,\dots,h_{j_n},a_n)$ is an $\alpha$-adapted list.
For that observe that $[b_{j_i-1},b_{j_i}]\subseteq [a''_{i-1},a'_i]$ and $\beta((b_{j_i-1},b_j))$ is contained in $h_{j_i}(P)$.
Therefore $C_i\subseteq h_{j_i}(P)$, by (\ref{Closed}).

It remains to prove that $\Phi((\mathcal{L}')_i)=\Phi((\mathcal{D})_i)$ for every $i\in \{1,\dots,n\}$. We argue by induction.
As $\alpha(0)\in g_1(P)^\circ$, necessarily $C_1=g_1(P)$.
Using (\ref{Interior}) it is easy to prove that $\beta([a''_0=0,b_{j_1}=a'_1])\subseteq g_1(P)$.
Hence $h_j= g_1$ for every $j\in \{1,\dots,j_1\}$. Therefore $\Phi((\mathcal{L}')_1)=1=\Phi((\mathcal{D})_1)$.
Assume that $i>1$ and $\Phi((\mathcal{L}')_{i-1})=\Phi((\mathcal{D})_{i-1})$.
Let $E$ be the cell generated by $\alpha(a_{i-1})$, i.e. the unique one whose relative interior contains $\alpha(a_{i-1})$.
For every $j\in \{ 1,\dots,m\}$ let
$E_j$ be the cell generated by $\beta(b_j)$.
Then $E\subseteq C_i$, $\alpha(a_{i-1})\in E_j$ for every $j_{i-1}\le j \le j'_{i-1}$, by (\ref{Neighbour}) and $C_i$ is contained in every cell intersecting some $\beta([a''_{i-1},a'_i])$, by (\ref{Closed}).
Therefore, for every $j'_{i-1}\le j \le j_i$, $C_i$ is contained in every tile containing $\beta(b_j)$ and hence $E\subseteq C_i \subseteq E_j$.
We conclude that $E\subseteq E_j$ for every $j_{i-1}\le j \le j_i$.
Using that $E_j\subseteq h_j(P)\cap h_{j+1}(P)$ we have
	\begin{eqnarray*}
	\kappa_{\alpha(a_{i-1})}(h_{j_{i-1}},h_{j_i}) &=& \kappa_E(h_{j_{i-1}},h_{j_i}) =
	\kappa_E(h_{j_{i-1}},h_{j_{i-1}+1}) \cdots \kappa_E(h_{j_i-1},h_{j_i}) \\
	&=&
	\kappa_{E_{j_{i-1}}}(h_{j_{i-1}},h_{j_{i-1}+1}) \cdots \kappa_{E_{j_i-1}}(h_{j_i-1},h_{j_i}) \\
	&=& \kappa_{\beta(b_{j_{i-1}})}(h_{j_{i-1}},h_{j_{i-1}+1}) \cdots \kappa_{\beta(h_{j_i-1})}(h_{j_i-1},h_{j_i})
	\end{eqnarray*}
by Lemma~\ref{kappaElementaryL}.(\ref{kappaTransitiveCell}).
Then
\begin{eqnarray*}
\Phi((\mathcal{L}')_i)&=&\Phi((\mathcal{L}')_{i-1})\kappa_{\alpha(a_{i-1})}(h_{j_{i-1}},h_{j_i})\\
&=&\Phi((\mathcal{D})_{i-1})\kappa_{\beta(b_{j_{i-1}})}(h_{j_{i-1}},h_{j_{i-1}+1}) \cdots \kappa_{\beta(b_{j_i-1})}(h_{j_i-1},h_{j_i}) \\
&=&\Phi((\mathcal{D})_i),
\end{eqnarray*}
as desired.
\end{proof}

\begin{proof} ({\bf of Theorem~\ref{PoincareT} (Relations))}.
Let $g_1,\dots,g_n$ be a list of pairing transformations such that $g_1\cdots g_n=1$. We have to show that $[g_1]\cdots [g_n]=1$.
We may assume that $g_1 \cdots g_n=1$ is a minimal relation, i.e. $g_i \cdots g_j\ne 1$ for every $1\le i < j \le n$ with $(i,j)\ne (1,n)$.
For every $i\in \{1,\dots,n\}$ let $S_i=g_1\cdots g_{i-1}(P)\cap g_1 \cdots g_i(P)$, a side of both $g_1\cdots g_{i-1}(P)$ and $g_1 \cdots g_i(P)$.

Fix $c_0\in P^{\circ}$ and for every $i\in \{1,\dots,n\}$ let $b_i\in (P\cap g_i(P))^r$ and $c_i =g_1 g_2 \cdots g_{i-1} (b_i )$.
Observe that each $c_i\in S_i$ for $i\geq 1$.
If the sides $S_{i}$ and $S_{i+1}$ are contained in different essential hyperplanes of the tile
$g_1 g_2 \cdots g_i (P)$ then $(c_i , c_{i+1}) \subseteq (g_1 g_2 \cdots g_i )(P)^{\circ}$.
However, they might be in the same essential hyperplane, and therefore
we introduce some additional elements of $\XQ$.
For each $i\in \{ 1, 2, \ldots , n-1\}$, choose $c_i'\in g_1 g_2 \cdots g_i (P)^{\circ}$.
Consider the geodesic segments,
  $$
	[c_0,c_1=b_1],\;  [c_1=b_1,c_1'],\; [c_1', c_2=g_1(b_2)],\;
  	  [c_2,c_2'], \; [c_2',c_3=g_1g_2(b_3)],\dots ,
	$$
	$$
	  [c_{n-1}=g_1 \cdots g_{n-2} (b_{n-1}),c_{n-1}']\; [c_{n-1}', c_{n}=g_1 g_2 \cdots g_{n-1}(b_n)],\;
	  [c_n,c_{n+1}=c_0].
	$$
By construction, for each $1\leq i<n$,
  $$(c_0 ,c_1)\subseteq P^{\circ}, (c_i ,c_i')\subseteq g_1 \cdots g_i (P)^{\circ},\\
    (c_{i}',c_{i+1})\subseteq g_1 \cdots g_i (P)^{\circ}, (c_{n},c_{0})\subseteq P^{\circ}.$$
Furthemore, the closure of each of the listed geodesic segments is contained in $\XQ_{\te}$.
Let $\alpha : [0,1]\rightarrow \XQ_{\te}$ be the continuous function whose graph is obtained by concatenating all these segments.
Then there exists an ascending sequence
  $$a_0 < a_1 < a_1' < a_2 < a_2' <a_3 < \cdots < a_{n-1} < a_{n-1}' < a_{n} <a_{n+1}=1$$
with $\alpha (a_i) =c_i$ and $\alpha (a_i') =c_i'$ for each $i$. Furthermore,
 $$\alpha ((a_{i-1},a_i')) =(c_i , c_i') \subseteq g_1 \cdots g_i (P)^{\circ} \quad \text{ and } \quad
   \alpha ((a_i',a_{i+1}))=(c_i',c_{i+1})\subseteq g_1 \cdots g_i (P)^{\circ}.$$
Therefore,
 \begin{eqnarray*}
 \mathcal{L} &=&(a_0, 1, a_1 , g_1 , a_1', g_1 ,a_2, g_1 g_2, a_2', g_1 g_2, a_2, g_1 g_2 g_3 a_3' , \\
  && \;\;  \cdots , g_1 g_2 \cdots g_{n-1} , a_{n-1}', g_1\cdots g_{n-1} a_n , g_1 \cdots g_n , a_{n+1})
  \end{eqnarray*}
is an $\alpha$-adapted list.
Thus,
 \begin{eqnarray*}
 \Phi(\alpha)&=& [g_1]\,  [g_1\inv g_1]\, [ g_1\inv  (g_1g_2)]\,
      [(g_1g_2)\inv (g_1 g_2)]\,   [(g_1g_2)\inv (g_1 g_2 g_3)]\\
      && \;\; \cdots
       [(g_1\cdots g_{n-1})\inv (g_1\cdots g_n-1)]\,
        [(g_1\cdots g_{n-1})\inv (g_1\cdots g_n)]\\
      & =& [g_1][g_2]\cdots [g_n] .
 \end{eqnarray*}
On the other hand,  let $\beta\in \mathcal{C}_{c_0,c_0}([0,1],\XQ_{\te})$ denote  the constant path, i.e. $\beta(t)=c_0$ for every $t\in [0,1]$.
Then $(0,1,1)$ is a $\beta$-adapted list and $\Phi(\beta)=1$.
Lemma~\ref{PhiConstantL} yields that
   $[g_1]\cdots [g_n]=\Phi(\alpha)=\Phi(\beta)=1$, as desired.
\end{proof}

\bibliography{ReferencesMSC}

\end{document}